\documentclass[a4paper, reqno, 11pt]{amsart}
\usepackage[utf8]{inputenc}
\usepackage{amsmath}
\usepackage{amssymb}
\usepackage{amsthm}
\usepackage[dvipsnames]{xcolor}
\usepackage{comment}

\newtheorem{theorem}{Theorem}[section]
\newtheorem{corollary}[theorem]{Corollary}
\newtheorem{lemma}[theorem]{Lemma}

\newtheorem{proposition}[theorem]{Proposition}
\theoremstyle{definition}

\theoremstyle{remark}
\newtheorem*{remark}{Remark}

\DeclareMathOperator\supp{supp}

\usepackage[luatex,colorlinks=true, linkcolor=WildStrawberry, citecolor=PineGreen]{hyperref}



\begin{document} 
	
	\title{Pointwise convergence of the Klein-Gordon flow}
	
	\author[R. Lucà]{Renato Lucà}
	\address[R. Lucà]{Basque Center for Applied Mathematics, 48009 Bilbao, Basque Country and Ikerbasque, Basque Foundation for Science, 48011 Bilbao, Basque Country}
	
	\author[P. Merino]{Pablo Merino}
	\address[P. Merino]{Basque Center for Applied Mathematics, 48009 Bilbao, Basque Country}
	
	\date{}

	\maketitle
	
	\begin{abstract}
		We consider the PDEs version of the Carleson problem in the context of the cubic nonlinear Klein-Gordon equation. This 
		means that we aim to establish the lowest regularity class for which one has almost 
		everywhere pointwise convergence  of the solutions to the initial data, as $t \to 0$. We prove sharp results for initial 
		data in Sobolev spaces and for their randomized counterparts. 
	\end{abstract}
	
	\tableofcontents

	\section{Introduction}
	We consider the Cauchy problem associated to the 3d periodic Klein-Gordon 
	equation with a cubic nonlinearity, that is  	
	\begin{equation}\label{nlwcauchy}
		\left\{
		\begin{array}{l}
			\partial_{tt} u - \Delta u + u \pm u^3 = 0, 
			\\
			u(x,0) = u_0(x),  \ \ \partial_tu(x,0) = u_1(x).
		\end{array}
		\right.
	\end{equation}
	Here the solution is a real valued periodic function $u : \mathbb{T}^3 \times [0,\delta] \to \mathbb{R}$ and the initial datum is assigned 
	at the time $t=0$. Our results will be independent on the focusing or defocusing nature of the model. We are interested in the 
	pointwise behavior of the solutions and, more precisely, we will address the following 
	question: which is the 
	minimal regularity $s$ such that 
	$$u(t,x) \rightarrow u_0(x)  \quad \mbox{as} \quad t \rightarrow  0
	$$ 
	for  (Lebesgue) almost every $x \in \mathbb{T}^3$ and 
	for all $(u_0,u_1) \in H^s(\mathbb{T}^3) \times H^{s-1}(\mathbb{T}^3)$~?	
	Moreover, we will show that improved (pointwise) convergence results are available if we consider a suitable randomization 
	of the initial datum. 
	
	Before stating our results, we do a brief excursus on the pointwise convergence problem in the 
	context of dispersive PDEs. This problem was introduced by Carleson \cite{Carleson} in the 80s, for the linear Schr\"odinger 
	equation. He was interested in identifying the minimal Sobolev regularity of initial data 
	$f \in H^{s}(\mathbb{R}^d)$ for which one has
	$e^{it \Delta} f (x) \overset{t \to 0}{\to} f(x)$ for almost every $x \in \mathbb{R}^d$.   
	The 1d problem was solved by Carleson \cite{Carleson}, who showed that this (a.e.) pointwise convergence holds for $f \in H^{\frac{1}{4}}(\mathbb{R})$, and Dalhberg-Kenig \cite{DahlbergKenig}, who showed that $s \geq \frac{1}{4}$ is necessary on $\mathbb{R}^d$, $d \geq 1$. The higher dimensional case, on the other hand, has been only recently solved in \cite{Bourgain2016, DGL, Du2019} and the final solution required very deep
	harmonic analysis techniques. Namely, for any $d \geq 1$, $s \geq \frac{d}{2(d+1)}$ is necessary, and $s  > \frac{d}{2(d+1)}$ is sufficient up to the endpoint. We refer to \cite{MR904948, Vega, MR1194782, MR1671214, MR1748920, MR1748921, MR2033842, MR2264734, Bourgain2013, Luc2018, MR3613507, MR3903115, 1608.07640, MR3842310}
	and the references therein for intermediate results on this matter. It is 
	also worth to point out that the periodic problem is still open (in any dimension) and we refer 
	to \cite{MR2409184, WangZhang, EL21} for the most recent 
	results. In particular, $s > \frac{d}{d+2}$ is known to be sufficient, and $s \geq \frac{d}{2(d+1)}$ to be necessary. At the moment, in the periodic case almost sure convergence when $s \in [ \frac{d}{2(d+1)} , \frac{d}{d+2} ]$ remains an open problem. In the case of the Schr\"odinger operator, the nonlinear version of the 
	pointwise convergence problem (as well as its probabilistic counterparts) 
	has been studied by the first author, E. Compaan and G. Staffilani in \cite{CSL}. 
	
	As mentioned, the goal of this paper is 
	to develop a nonlinear  pointwise convergence theory in the case of the Klein-Gordon equation. To fix the ideas and avoid 
	unimportant technicalities, we will focus on the 3d cubic model \ref{nlwcauchy}, 
	that contains all of the interesting features of the problem. We will also focus on the periodic case, which does not present extra difficulties compared with the continuous one. One can indeed see that in the proof that we give below of the 
	sufficiency of the (optimal) condition $s > 1/2$ we can replace $\mathbb{T}^3$ with $\mathbb{R}^3$ without any change. We point out 
	that the same argument for the Schr\"odinger equation does not lead to the optimal regularity threshold. In the linear setting, it was proved in \cite{C831, C832, Walther} that
	one has $e^{it \langle D \rangle} f (x) \overset{t \to 0}{\to} f(x)$ for almost every $x \in \mathbb{T}^d$ if and only if 
	$s > 1/2$ (in any dimension). Our first objective is to prove a nonlinear counterpart of this result.
	\begin{theorem}\label{MainThm1}
		Let $s > 1/2$ and 
		$(u_0, u_1) \in H^{s}(\mathbb{T}^3) \times H^{s-1}(\mathbb{T}^3)$. 
		Let $u$ the (local in time) solution of the 
		Cauchy problem \eqref{nlwcauchy}. We have 
		$$ \lim_{t \to 0}
		|u(t,x) - u_0(x)| = 0, \quad \mbox{for almost every $x \in \mathbb{T}^3$} .
		$$ 
	\end{theorem}
	
	The statement is optimal for the following (two different) reasons. First, if $s < 1/2$ we do not have 
	local well-posedness for the 
	Cauchy problem~\eqref{nlwcauchy}; \cite{BurqTzvet}. Second, when $s \leq 1/2$ the a.e. pointwise convergence of 
	the solutions to the linearized problem to their initial data fails too; \cite{Walther}.

	In the second part of the paper we aim to lower the regularity assumption essentially up to
	$s>0$. As observed, this is impossible in the deterministic setting, since the statement is sharp. However, 
	we will see that this is possible considering a suitable randomization of the initial data, in the spirit of \cite{BurqTzvet}.  
	Thus, given  $(u_0, u_1) \in H^{s}(\mathbb{T}^3) \times H^{s-1}(\mathbb{T}^3)$ we randomize the Fourier coefficients as follows
	\begin{align}\label{RandomInitialData}
		u_0^w = \sum_{n \in \mathbb{Z}^3} h^w_n \hat{u}_0(n) e^{in \cdot x}, \quad 
		u_1^w = \sum_{n \in \mathbb{Z}^3} l^w_n \hat{u}_1(n) e^{in \cdot x},
	\end{align}
	where  
	$(h_n^w)_{n \in \mathbb{Z}^3}$, $(l_n^w)_{n \in \mathbb{Z}^3}$ are two sequences of 
	identically distributed pairwise independent sub-Gaussian random variables. 
	One can prove (see \cite{BurqTzvet}) that $w$-almost surely the following holds:
	$(u_0^w, u_1^w) \in H^{s}(\mathbb{T}^3) \times H^{s-1}(\mathbb{T}^3)$
	but  $(u_0^w, u_1^w) \notin H^{s + \varepsilon}(\mathbb{T}^3) \times H^{s-1 + \varepsilon}(\mathbb{T}^3)$ (for all $\varepsilon >0$).
	Moreover, for all $s' < s$ one has $(u_0^w, u_1^w) \in C^{s'}(\mathbb{T}^3) \times C^{s'-1}(\mathbb{T}^3)$ and $(u_0^w, u_1^w) \in W^{s,p}(\mathbb{T}^3) \times W^{s-1,p}(\mathbb{T}^3)$ for all $p \in [1,\infty)$, again 
	$w$-almost surely. Loosely speaking, the latter means that the randomization 
	improves the $L^p$ integrability, while the former means that it does not improve the regularity of the original functions. 
	
	In \cite{BurqTzvet} the authors proved that the Cauchy problem 
	\begin{equation}\label{nlwcauchyProb}
		\left\{
		\begin{array}{l}
			\partial_{tt} u^w - \Delta u^w + u^w \pm (u^w)^3 = 0, 
			\\
			u^w(x,0) = u^w_0(x),  \ \ \partial_t u^w(x,0) = u^w_1(x).
		\end{array}
		\right.
	\end{equation}
	is $w$-almost surely 
	locally well posed, for all $s>0$. This is particularly interesting since one can not establish 
	any deterministic local well-posedness result at this level of regularity (indeed one must require $s\geq1/2$). A natural question 
	is thus to establish whether for these solutions we have almost everywhere pointwise convergence to the initial data. This 
	is proved in our second main result, namely Theorem \ref{MainThm2}, that extends Therorem \ref{MainThm1} in the probabilistic sense, to the 
	whole range of regularity $s>0$. 
	
	\begin{theorem}\label{MainThm2}
		Let $s > 0$ and 
		$(u_0^w, u_1^w)$ as in \eqref{RandomInitialData}. For $w$-almost every initial datum $(u_0^w, u_1^w)$  
		we have  
		$$ \lim_{t \to 0}
		|u^w(t,x)  - u^w_0(x)| = 0, \quad \mbox{for almost every $x \in \mathbb{T}^3$},
		$$ 
		where $u^w$ is the solution of the Cauchy problem \eqref{nlwcauchyProb} constructed in \cite{BurqTzvet}.
	\end{theorem}
	The rest of the paper is organized as follows. In Section \ref{Sec:SOTP} we describe informally the strategy 
	of the proofs of our main theorems. In Section \ref{Sec:TCT} we recall the Cowling theorem for the a.e. 
	pointwise convergence of the half wave operator $e^{it \langle D \rangle}$. This linear result will be also 
	fundamental in the nonlinear analysis. In Section \ref{sec:TMATNPC} we adapt the Maximal estimates method 
	(from the linear) to the nonlinear 
	framework. This is the main idea behind the proofs and it is indeed worth 
	to stress out that maximal estimates 
	are the strongest tool 
	at our disposal to deduce pointwise almost everywhere convergence.   In Section \ref{Sec:DetPre} we recollect some 
	preliminary results of the local well posedness that will be used in the proof of Theorem \ref{MainThm1}, 
	that is done in Section \ref{Sec:Proof1}.
	In Section \ref{Sec:ProbPre} we recollect some 
	preliminary results of the local well posedness that will be used in the proof of Theorem \ref{MainThm2}, 
	that is done in Section \ref{Sec:Proof2}. 
	
	\section{Notations}
	\begin{itemize} 
		\item
		Given $x \in \mathbb{R}^d$, we abbreviate $\langle x \rangle := (1 + |x|^2)^{1/2}$. We will denote with~$\langle D \rangle$ the Fourier multiplier with symbol $\langle \xi \rangle$, i.e. $\langle D \rangle = \sqrt{1-\Delta}$.
		\item We 
		abbreviate  
		$L^p_{\delta}L^q = L^p_t([0,{\delta}],L^q_x(\mathbb{T}^3))$.
		\item Given $X$ and $Y$ two Banach spaces with norms $\| \cdot \|_X$ and
		$\| \cdot \|_Y$ respectively, we endow the product space $X \times Y$ with the 
		norm  $\| \cdot \|_{X \times Y} = \| \cdot \|_X +  \| \cdot \|_Y$.
		\item We denote with $X^{s,b}_{\delta,+} \times X^{s,b}_{\delta,-}$  
	the restriction space adapted to the linear flows $e^{\mp i t \langle D \rangle}$. We thus define
	\begin{align}\label{RestrSpaceDef}
		\|F\|_{X^{s,b}_{\delta,\pm}} = \inf_{G = F \text{ on } t \in [0,\delta]} \|G\|_{X^{s,b}_{\pm}},
	\end{align}
	where 
	\begin{align*}
		\|F\|^2_{X^{s,b}_{\pm}} = \int_{\mathbb{R}} \sum_{n \in \mathbb{Z}^3} \langle \tau \pm \langle n \rangle \rangle^{2b} \langle n \rangle^{2s} |\widetilde{F}(n,\tau)|^2 d\tau,
	\end{align*}
	and $\widetilde{F}$ is the space-time Fourier transform of $F$.  There is a useful way of characterizing this restriction space norms through mixed Sobolev norms, since
	\begin{align}
		&\|F\|_{X^{s,b}_{\pm}} =  \| e^{\pm i t \langle D \rangle} F\|_{H^s_x(\mathbb{T}^3)H^b_t(\mathbb{R})}.
		\label{RestrictionSobolevChar}
	\end{align}		
		\item
		Given $R>0$, we denote $B(0,R)$ the balls centered
		in $0$ with radius $R$ within the space $X^{s,b}_{\delta,+} \times X^{s,b}_{\delta,-}$, 
		and $B_{\pm}(0,R)$ the balls centered in $0$ with radius $R$ within the space $X^{s,b}_{\delta,\pm}$. 
		\item
		$P_{\leq N}$ will be the frequency projection on the ball of radius $N$, centered in the origin; as well as $P_N = P_{\leq N} - 	P_{\leq N/2}$ and $P_{>N} = Id - P_{\leq N}$. In this sense, when nothing else is said, $N$ will denote an arbitrary element of a nonnegative increasing sequence diverging to~$\infty$.
		\item When we apply a $\sup$ or a $|\cdot|$ to a vector of functions, we are applying it componentwise. For instance, for functions $f : A \rightarrow \mathbb{R}$ and $g:A \rightarrow \mathbb{R}$, then
		\begin{align*}
			\sup_{x \in A}|(f(x) , g(x))| = (\sup_{x \in A} |f(x)| , \sup_{x \in A} |g(x)| ).
		\end{align*}
		The same will be considered with projection operators from the previous point. 		
	\end{itemize}
	\section{Strategy of the proof}\label{Sec:SOTP}
	
	Assume, to fix the ideas, that we are in the focusing case.
	We will actually rewrite the cubic Klein-Gordon equation \eqref{nlwcauchy} as 
	\begin{equation}\label{nlwdec}
		\left\{
		\begin{array}{l}		
			( \partial_t \pm i \langle D \rangle ) u_{\pm} = \pm i \langle D \rangle^{-1}\left( \left( \frac{u_+ + u_-}{2} \right)^3 \right), 
			\\
			u_{\pm}(x,0) = u_{0,\pm}(x), 
		\end{array}
		\right.
	\end{equation}
	where 
	\begin{align}
		u_{\pm} = u \pm i \langle D \rangle^{-1}\partial_t u, \quad u_{0,\pm} =  u_0(x) \pm i \langle D \rangle^{-1} u_1(x).
		\label{decic}
	\end{align}

	In practice, we will actually work with 
	the corresponding 
	Duhamel formulation of the problem
	\begin{align}\label{ourflow}
		\begin{pmatrix} u_+(t,x) \\  \\ u_-(t,x) \\\end{pmatrix} =  \begin{pmatrix} 
			e^{-it\langle D \rangle} u_{0,+} + i\int_0^t \frac{e^{-i(t-\tau)\langle D \rangle}}{\langle D \rangle}  \left( \frac{u_+(\tau) + u_-(\tau)}{2} \right)^3  d\tau
			\\
			\\
			e^{it\langle D \rangle} u_{0,-} - i\int_0^t \frac{e^{i(t-\tau)\langle D \rangle}}{\langle D \rangle} \left( \frac{u_+(\tau) + u_-(\tau)}{2} \right)^3  d\tau	\\\end{pmatrix}.
	\end{align}
	
	Note that $(u_0,u_1) \in H^s(\mathbb{T}^3) \times H^{s-1}(\mathbb{T}^3)$ if 
	and only if $u_{0,\pm} \in H^s(\mathbb{T}^3)$, thus we have reduced Theorem \ref{MainThm1}
	to the following.
	\begin{theorem}\label{MainThm1Bis}
		Let $s>1/2$ and $(u_{0,+}, u_{0,-}) \in H^s(\mathbb{T}^3) \times H^s(\mathbb{T}^3)$. We have	
		\begin{equation}\label{MainThm1BisEstimate}
			\lim_{t \to 0} |u_{\pm}(x, t) - u_{0,\pm}(x) | = 0, \quad \mbox{for almost every $x \in \mathbb{T}^3$}.
		\end{equation}
	\end{theorem}
	In the linear setting, the most powerful tool to prove almost everywhere pointwise convergence 
	is the maximal estimate approach. More precisely, one can deduce (this is classic)
	$$
	\lim_{t \to 0} |e^{it\langle D \rangle} f(x) - f(x) | = 0, 
	\quad \mbox{for almost every $x \in \mathbb{T}^3$ and for all $f \in H^s(\mathbb{T}^3)$}
	$$ 
	from the maximal estimate  
	\begin{equation}\label{mfdklsjhgfhjskdngkls}
		\left\| \sup_{0 <  t < \delta}  |e^{it\langle D \rangle} f(\cdot) | \right\|_{L^2(\mathbb{T}^3)}
		\lesssim \| f \|_{H^{s}(\mathbb{T}^3)}, \qquad \delta >0.
	\end{equation}
	Proving such a maximal estimate is more amenable than attacking the pointwise convergence problem 
	directly.  
	The estimate is actually known to be true for $s>1/2$ (see \cite{C831}), setting $\delta = 1$. However, we are allowed to assume that the constant 
	$\delta >0$ in \eqref{mfdklsjhgfhjskdngkls} may depend on $f$.
	
	One would be then tempted to deduce the (nonlinear) pointwise convergence Theorem
	\ref{MainThm1Bis} by an analogous (nonlinear maximal estimate) 
	\begin{equation}\label{eq:MaxEst_NL}
		\left\| \sup_{0 < t < \delta} |u_{\pm}(\cdot, t) | \right\|_{L^2(\mathbb{T}^3)} \lesssim 
		\left\|  u_{0,\pm}  \right\|_{H^s(\mathbb{T}^3)}.
	\end{equation}
	Unfortunately, it is not true that the estimate \eqref{eq:MaxEst_NL} implies the (a.e.) pointwise convergence. 
	However, a rigorous nonlinear analogous 
	of \eqref{mfdklsjhgfhjskdngkls} can be established  and one can use it to 
	deduce \eqref{MainThm1BisEstimate}. This rigorous analogous is given by equation 
	\eqref{equivconvergence} in 
	Lemma \ref{lemma:AdaptMaxEst_NL}. 
	
	For the sake of simplicity, let us explain how to estimate the l.h.s. of \eqref{eq:MaxEst_NL}. The argument that 
	will allow us to prove the (desired) equation 
	\eqref{equivconvergence} is essentially the same.  
	The key idea is to use the fact that the  
	restriction spaces   $X^{s, b}_{\delta,\pm}$ (defined in \eqref{RestrSpaceDef}) are embedded into the maximal space 
	$L^{2}(\mathbb{T}^3; L^{\infty}(0, \delta))$  
	as long as $s >1/2$ and $b >1/2$.   
	Thus, an efficient estimate of the l.h.s. of  \eqref{eq:MaxEst_NL} is achieved using an inequality like
	$$
	\left\| \sup_{0 < t < \delta} |u_{\pm}(\cdot, t) | \right\|_{L^2(\mathbb{T}^3)} \lesssim 
	\left\| u_{\pm} \right\|_{X^{s,b}_{\delta,\pm}}, \qquad s, b > 1/2,
	$$
        and the local well-posedness theory 
	(in the restriction spaces framework). This heuristic is formalized in Section \ref{Sec:Proof1}.
	
	Regarding the probabilistic improvement, the strategy starts with a split between homogeneous (linear flow) and non-homogeneous (Duhamel formula) terms for the solution, in the spirit of \cite{BurqTzvet}. Namely,
	\begin{align}
		&(\partial_t \pm i\langle D \rangle)z_{\pm}^w = 0, \quad z_{\pm}^w(x,0) = u_{0,\pm}^w(x), \label{descomp_linear}\\
		&(\partial_t \pm i\langle D \rangle)v_{\pm} = \pm i \langle D \rangle^{-1} \left(  \frac{z^w + v}{2} \right)^3, \quad v_{\pm}(x,0) = 0,
		\label{descomp_nonlinear}
	\end{align}
	where $v = v_+ + v_-$ and $z^w = z^w_+ + z^w_-$, and where initial datum $(u_{0, +},u_{0, -})$ before randomization is in $H^{\sigma}(\mathbb{T}^3) \times H^{\sigma}(\mathbb{T}^3)$, for some $\sigma > 0$.
	
	Then we perform a direct analysis to show that $\omega$-almost surely the linear part of the flow converges to the initial datum almost everywhere (in fact everywhere and uniformly). This 
	only requires to take advantage of the 
	sub-Gaussian nature of the randomization. 
	Regarding the nonlinear contribution, we will reduce to prove, as before, a suitable rigorous version of
	the informal maximal bound\footnote{Note that $v_{\pm}$ depends on $w$ even if we are not stressing this out in the notation.}
	$$
	\| \sup_{0 < t < \delta_w} \left| v_{\pm}  | \hspace{1mm} \right\|_{L^2(\mathbb{T}^3)} \lesssim 
	\| v_{\pm}  \|_{X^{s,b}_{\delta_w,\pm}} , \qquad s, b >1/2.
	$$	
	Now the important point is that the above inequality requires $s > 1/2$ regularity (in the scale of the restriction spaces), while we are working with initial data that are,
	$w$-almost surely, just in $H^{\sigma}$ and we only assumed $\sigma >0$. 
	We will be able to avoid this issue using (crucially) the probabilistic smoothing of the 
	Duhamel contribution $v_{\pm}$ 
	(more precisely, we need at least 
	$\frac12-\varepsilon$ smoothing, for some $\varepsilon >0$, in order to prove the key 
	equation \eqref{eq:12epsilon-smoothing} below).  
	In order to take advantage of the smoothing effect in our framework we will adapt to the restriction 
	space setting the argument from \cite{BurqTzvet}. 
	This will allow us to work just above the $L^2$ level (recall $\sigma >0$).	
	
	\begin{remark}
	The method used in this paper presumably adapts to higher order nonlinearities. For instance, 
	one can focus on $N(u) = |u|^{\alpha-1}u$, $\alpha \geq 3$. 
	For this nonlinearity the critical scaling for local well-posedness is
	$s_c = \frac{3}{2} - \frac{2}{\alpha-1}$. Assuming 
	that one can prove LPW for $s > s_c$ (that is for instance the case when $\alpha \in [3, 5)$, see 
	Theorem 1.1 of \cite{SX15}), we expect that 
	the techniques developed in this paper would allow to deduce pointwise convergence to the initial data in 
	the deterministic setting in the sub-critical regime. Regarding the super-critical regime, if we consider initial 
	data obtained from 
	the randomization of an $H^{\sigma}$ function, we guess that our method would imply
	almost sure almost everywhere (a.s., a.e.) pointwise convergence 
	to the initial data as long as 
	$$\sigma + \gamma > \max(s_c, 1/2),$$
	where $\gamma$ is the (probabilistic) smoothing of the Duhamel contribution; see Theorem \ref{Prop:BT} 
	and the remark 
	after it. 
	Note that the second condition in the maximum is needed to use the Cowling convergence theorem, while the first 
	one is necessary in order to have probabilistic local well-posedness. Since for $\alpha \geq 3$ one has $s_c \geq 1/2$, we then expect that
	the method that we have developed would show that in this case one can deduce
	a.s., a.e. convergence as long as the problem is locally well-posed in the probabilistic 
	sense. Thus, for instance, combining the techniques of this paper 
	and \cite{SX15} one would be able to extend the analysis to 
	the range $\alpha \in [3, 5)$ (modulo adapting the result from \cite{SX15} to the restriction space 
	framework that we used here).  For larger values of $\alpha$ 
	there is certainly hope to 
	extend the (a.s., a.e.) pointwise convergence result using more advanced tools recently developed in the study of 
	probabilistic PDEs (for the wave equation, we refer to \cite{OP16, OP17, BDNY24}	 and the references therein). 
	Combining these refined 
	tools with the techniques that we developed in this paper is however more delicate and we will not pursue 
	this matter any further here.    
	\end{remark}
	

	\section{The Cowling convergence theorem}\label{Sec:TCT}
	
	We include now a result of M. G. Cowling \cite{C831, C832}. We will state and use a less general form of his theorem, so 
	we also provide a sketch of the proof for the reader's convenience. 
	\begin{theorem}\label{thm:cowling}
		Let $s > 1/2$. Then   
		\begin{equation}\label{cowlingLemma}
			\left\| \sup_{0 < t < 1} |e^{i t \langle D \rangle} f(\cdot) | \right\|_{L^2(\mathbb{T}^d)} \leq C_s
			\left\|  f(\cdot)  \right\|_{H^s(\mathbb{T}^d)} .
		\end{equation}
	\end{theorem}
	\begin{proof}
		Invoking the Littlewood-Paley decomposition, it is sufficient to prove
		$$
		\left\| \sup_{0 < t < 1} |e^{i t \langle D \rangle} f(\cdot) | \right\|_{L^2(\mathbb{T}^d)} \lesssim
		N^{\frac12}\left\|  f(\cdot)  \right\|_{L^2(\mathbb{T}^d)} , \qquad N \geq 1,
		$$
		for all $2 \pi$-periodic function $f$ with $\supp \widehat f \subset [-N, N]^d$. Let $F(t) = e^{it\langle D \rangle} f(x)$ and $\alpha > 0$ be a parameter that will be chosen later. Notice that $\partial_tF(t) = i \langle D \rangle e^{it\langle D \rangle} f(x)$. By fundamental theorem of calculus and Hölder's inequality we obtain that
		\begin{align*}
		&|F^2(t)| = |F^2(0) + \int_0^t \partial_y (F^2(y)) dy| \leq |F(0)|^2 + 2 \int_0^t \left|(\partial_yF(y)) F(y) \right| dy \nonumber \\
		&\leq |F(0)|^2 + 2 \frac{1}{\sqrt{\alpha}} \| \partial_y F \|_{L^2(0,t)} \sqrt{\alpha} \|F\|_{L^2(0,t)} \nonumber \\
		&\leq |F(0)|^2 + \frac{1}{\alpha}\| \partial_y F \|_{L^2(0,t)}^2 + \alpha \|F\|_{L^2(0,t)}^2.
		\end{align*}
		This inequality is satisfied for any $t \in (0,1)$. Thus
		\begin{align*}
		\sup_{0<t<1} |e^{it\langle D \rangle} f(x)|^2 \leq  |f(x)|^2 + 
		\frac{1}{\alpha} \|\langle D \rangle e^{it\langle D \rangle} f(x)\|_{L^2(0,1)}^2 
		+ \alpha \| e^{it\langle D \rangle} f(x) \|_{L^2(0,1)}^2 .
		\end{align*}
		If we integrate over $\mathbb{T}^d$ (namely, with respect to the $x$ variable) and we use Fubini and 
		$$
		\| e^{it\langle D \rangle} f(x)\|_{L^2(\mathbb{T}^d)} = \| f \|_{L^2(\mathbb{T}^d)}, \qquad		
		\|\langle D \rangle e^{it\langle D \rangle} f(x)\|_{L^2(\mathbb{T}^d)} \lesssim N \| f \|_{L^2(\mathbb{T}^d)}, 
		$$
		 we arrive to
		\begin{align*}
		\| \sup_{0<t<1} | e^{it\langle D \rangle} f(\cdot)| \|_{L^2(\mathbb{T}^d)}^2  \lesssim \|f(\cdot)\|_{L^2(\mathbb{T}^d)}^2 + \frac{N^2}{\alpha} \|f(\cdot)\|_{L^2(\mathbb{T}^d)}^2 + \alpha \|f\|_{L^2(\mathbb{T}^d)}^2.
		\end{align*}
		Taking $\alpha = N$ (that equalizes the second and the third term on the right hand side) we have proved
		the desired bound 
		\begin{align*}
		\|\sup_{0<t<1} | e^{it\langle D \rangle} f(\cdot) | \|_{L^2(\mathbb{T}^d)} \lesssim N^{\frac{1}{2}} \|f(\cdot)\|_{L^2(\mathbb{T}^d)}.
		\end{align*}
	\end{proof}

	\section{The Maximal approach to nonlinear pointwise convergence}\label{sec:TMATNPC}
	
	In this section we provide the framework to study the pointwise convergence problem in the nonlinear setting. 
	The goal of this section is to show how the limit \eqref{equivconvergence}
	\begin{itemize}
		\item is sufficient to deduce almost everywhere pointwise convergence to the initial datum;
		\item can be deduced by an appropriate bound in the restriction spaces. 
	\end{itemize}
	
	In Lemma \ref{lemma:AdaptMaxEst_NL} we will see that \eqref{equivconvergence} is the 
	rigorous counterpart of the informal maximal estimate \eqref{eq:MaxEst_NL}. 
	
	Let $u_0=(u_{0,+},u_{0,-}) \in H^s(\mathbb{T}^3) \times H^s(\mathbb{T}^3)$. We denote with  $\Phi^N_t u_0 = (\Phi^{N,+}_t u_{0,+},\Phi^{N,-}_t u_{0,-})$ the flow $\Phi^N_t$ applied on $u_0$, and associated to the (frequency) truncated Cauchy problem 
	\begin{equation}\label{nlwdecTruncated}
		\left\{
		\begin{array}{l}		 
			( \partial_t \pm i \langle D \rangle ) u_{\pm} = \pm i  P_{\leq N} 
			\langle D \rangle^{-1}\left( \left( \frac{u_+ + u_-}{2} \right)^3 \right), 
			\\
			u_{\pm}(x,0) = P_{\leq N} u_{0,\pm}(x).
		\end{array}
		\right.
	\end{equation}
	We also use the notation 
	$\Phi_t u_0  := \Phi^{\infty}_t u_0$ for the actual (non truncated) flow 
	$\Phi_t$ applied on $u_0$. We will now show how to deduce pointwise convergence from a suitable nonlinear maximal estimate.
	\begin{lemma}\label{lemma:AdaptMaxEst_NL}
		Let $s \geq 0$, $f=(f_+, f_-) \in  H^{s}(\mathbb{T}^3) \times H^{s}(\mathbb{T}^3)$.  If
		for some $\delta  > 0$ we have
		\begin{align}
			\lim_{N \rightarrow \infty} 
			\left\| \sup_{0 < t < \delta} 
			| \Phi_t f   - \Phi^N_t f  | \right\|_{L^2_x(\mathbb{T}^3)\times L^2_x(\mathbb{T}^3)} = 0,
			\label{equivconvergence}
		\end{align}
		then $\Phi_t^{\pm} f_{\pm} \overset{t \to 0}{\to} f_{\pm}$ for almost every $x \in \mathbb{T}^3$.
	\end{lemma}
	\begin{proof}
		It suffices to prove it componentwise. First of all, we decompose 
		\begin{align*}
			|\Phi^{\pm}_t f_{\pm} - f_{\pm} | \leq |P_{>N} f_{\pm}| + 
			| \Phi_t^{N,\pm} f_{\pm} - P_{\leq N} f_{\pm}| + |\Phi^{N,\pm}_t f_{\pm}- \Phi^{\pm}_t f_{\pm}|.
		\end{align*}
		The second term in the RHS converges to $0$ for any $x \in \mathbb{T}^3$, since $P_{\leq N} f$ belongs to $H^r(\mathbb{T}^3)$ for any $r$, from which it is easy to prove that $\Phi_t^{N,\pm} f_{\pm} \rightarrow P_{\leq N} f_{\pm}$ as $t \downarrow 0$ uniformly in $x \in \mathbb{T}^3$. Indeed,
\begin{itemize}
			\item If $r > \frac{1}{2}$, $\Phi_t^{N,\pm}$ is a bounded map from $H^r(\mathbb{T}^3)$ to $C([0,\delta],H^r(\mathbb{T}^3))$ for $\delta$ sufficiently small (see \cite{GST19}).
			\item Consider $r > \frac{3}{2}$. Then, by Sobolev lemma,
			\begin{align*}
				\| \Phi_t^{N,\pm}f - P_{\leq N}f \|_{L^{\infty}_x(\mathbb{T}^3)} \lesssim \| \Phi_t^{N,\pm}f - P_{\leq N}f \|_{H^r(\mathbb{T}^3)} \rightarrow 0, \text{ for } t \rightarrow 0.
			\end{align*}
		\end{itemize}
		Then, applying triangle and Markov inequalities, for any $\lambda > 0$,
		\begin{align*}
			|\{ x & \in \mathbb{T}^3 : \limsup_{t \rightarrow 0} |\Phi^{\pm}_t f_{\pm} - f_{\pm}| > \lambda \}\}|  \\
			& \leq |\{ x \in \mathbb{T}^3 : \sup_{0 < t < \delta } 
			|\Phi^{N,\pm}_t f_{\pm} - \Phi^{\pm}_t f_{\pm}| > \frac{\lambda}{3} \}\}|  
			+ |\{ x \in \mathbb{T}^3 : |P_{>N} f_{\pm}| > \frac{\lambda}{3} \}\}|  \\ 
			& \lesssim  \frac{1}{\lambda^2} \left( \| \sup_{0 < t < \delta } 
			|\Phi^{N,\pm}_t f_{\pm} - \Phi^{\pm}_t f_{\pm}| \|^2_{L^2(\mathbb{T}^3)} 
			+ \| P_{>N}f_{\pm} \|^2_{L^2(\mathbb{T}^3)} \right),
		\end{align*}
		where $|\cdot|$ denotes the Lebesgue measure. 
		We can now take the limit as $N \to \infty$ of this inequality getting, for all $\lambda >0$,
		$$
		|\{ x  \in \mathbb{T}^3 : \limsup_{t \rightarrow 0} |\Phi^{\pm}_t f_{\pm} - f_{\pm}| > \lambda \}\}| = 0.
		$$
		Choosing $\lambda = 1/n$ and considering the measure of the  union over $n \in \mathbb{N}$ of the 
		above sets the statement follows. 
	\end{proof}
	Rather than proving \eqref{equivconvergence} directly, we will use the fact that a  
	suitable restriction space is embedded into  the maximal space. This is proved via the usual  transference
	principle, which proof can be adapted from the one in~\cite{tao2006nonlinear} (see Lemma 2.9, page 100).
	\begin{lemma}\label{TransferLemma}
		Let $b > \frac{1}{2}$ and let $Y$ be a Banach space of functions 
		$F : (t,x) \in \mathbb{R} \times \mathbb{T}^3 \mapsto \mathbb{C}$. 
		Let $\alpha \in \mathbb{R}$. Assume
		\begin{align*}
			\| e^{i \alpha t} e^{\pm i t \langle D \rangle} f(x) \|_Y \leq C \| f \|_{H^s(\mathbb{T}^3)},
		\end{align*}
		for some constant $C > 0$ independent on $\alpha \in \mathbb{R}$. Then, for any $F \in Y$,
		\begin{align*}
			\|F\|_Y \leq C \|F\|_{X^{s,b}_{\mp}}.
		\end{align*}
		\label{lemmaembed}
	\end{lemma}
	
	Recalling the linear maximal estimate from \eqref{cowlingLemma}, valid for $s>1/2$, 
	we can use Lemma \ref{TransferLemma} with $Y = L^{2}(\mathbb{T}^3; L^{\infty}(0, \delta))$, 
	so that we deduce 
	\begin{equation}
		\left\| \sup_{0 \leq t \leq \delta} | \Phi^{\pm}_t f_{\pm}  - \Phi^{N,\pm}_t f_{\pm} | \right\|_{L^2_x(\mathbb{T}^3)}  \leq C_s 
		\|  \Phi^{\pm}_t f_{\pm}  - \Phi^{N,\pm}_t f_{\pm} \|_{X^{s,b}_{\delta,\pm}}, \quad (b, s > 1/2, \, \delta >0).
		\label{eq:MaxEstXsb}
	\end{equation}
	Note that, doing so, we have reduced the problem of pointwise convergence to prove a suitable inequality in 
	restriction spaces, that is something definitely more amenable in the nonlinear setting.
	\\

	\section{Deterministic preliminaries}
	\label{Sec:DetPre}	
	
	In this section we collect and prove some results on the local well-posedness of \eqref{nlwdec} in the 
	framework of the restriction spaces \eqref{RestrSpaceDef}.
	These results will be used in Section \ref{Sec:Proof1} in order to prove our first main Theorem \ref{MainThm1}. 
	
	First, we recall the (local in time version of the) Strichartz estimate \cite{GST19}. Recall that we denote $L^p_{\delta}L^q = L^p_t([0,{\delta}],L^q_x(\mathbb{T}^3))$.
	\begin{lemma}\label{Strichartz_plow}
		Let $2 < p \leq \infty$ and $q$ such that $\frac{1}{p} + \frac{1}{q} = \frac{1}{2}$, $f \in H^{2/p}(\mathbb{T}^3)$. Then,
		\begin{align}
			\| e^{\pm i t \langle D \rangle}f \|_{L^p_1L^q}  \leq C \|f\|_{H^{2/p}(\mathbb{T}^3)}.
		\end{align}
	\end{lemma}
	Let $\delta \in (0,1)$. Following \cite{BurqTzvet}, we consider the generalized Strichartz spaces
	\begin{align}
		X^s_{\delta} =C([0,\delta],H^s(\mathbb{T}^3)) \cap \left( \bigcap_{(p,q) \text{ $s$-admissible}} L^p_{\delta}L^q \right)
		\label{StrichartzSpace}
	\end{align}
	for $s \in [0,1)$, and where $(p,q)$, $\frac2s \leq p \leq \infty$, is an $s$-admissible pair of exponents if and only if
	\begin{align*}
		\frac{1}{p} + \frac{3}{q}  \geq \frac{3}{2} - s.
	\end{align*}
	The norm is taken as the supremum of all norms involved in the intersection \eqref{StrichartzSpace}. The dual space of $X^s_{\delta}$ will be denoted by $Y^s_{\delta}$, and given by
	\begin{align}
		Y^s_{\delta} = \bigcup_{(p,q) \text{ } s- \text{admissible}} L^{p'}_{\delta}L^{q'}.
		\label{StrichartzSpaceDual}
	\end{align}
	Analogously to \eqref{StrichartzSpace}, its norm is taken as the infimum of all the norms involved in the union \eqref{StrichartzSpaceDual}.
	\\
	From Lemma \ref{Strichartz_plow}, Sobolev embedding and $H^s$ conservation we deduce (check \cite{GST19} for details)
	\begin{align}
		\| e^{\pm i t \langle D \rangle}f \|_{X^s_{\delta}}  \lesssim \|f\|_{H^s(\mathbb{T}^3)}, \qquad s \in (0,1).
		\label{LinearFlowStrichartz_slow}
	\end{align}
	Thus the transference Lemma \ref{lemmaembed} implies 
	\begin{align}
		\|F\|_{X^s_{\delta}} \lesssim \|F\|_{X^{s,b}_{\pm}},
		\label{EstimateStrichartzRestriction}
	\end{align}
	for any $F : \mathbb{R} \times \mathbb{T}^3 \mapsto \mathbb{C}$ in $X^{s,b}$, $\delta \in (0,1)$, $s \in (0,1)$ and $b > \frac{1}{2}$. 
	By duality and recalling \eqref{StrichartzSpaceDual}, this is equivalent to
	\begin{align}
		\|F\|_{X^{-s,-b}_{\pm}} \lesssim \|F\|_{Y^s_{\delta}},
		\label{EstimateStrichartzRestrictionDual}
	\end{align}
	for the same values of $\delta$, $s$ and $b$. It is clear that the same inequalities can be stated with the localized spaces $X^{s,b}_{\delta, \pm}$ and $X^{-s,-b}_{\delta, \pm}$.

	Hereafter $\eta$ will be a smooth, compactly supported, non negative cut-off of $[0,1]$. 
	We recollect some well known facts about restriction spaces (see \cite{ETBook}, Lemmata 3.10 and 3.12).
	\begin{lemma}
		\begin{align*}
			\|\eta(t)e^{\pm i t \langle D \rangle}f\|_{X^{s,b}_{\mp}} \lesssim_{\eta} \|f\|_{H^s(\mathbb{T}^3)}.
		\end{align*}
		\label{EstLinearFlowRestriction}
	\end{lemma}
	\begin{lemma}
		Let $b > \frac{1}{2}$ and $\delta \in (0,1)$. 
		\begin{align*}
			\|\eta(t) \int_0^t \frac{e^{\pm i (t - \tau) \langle D \rangle}}{\langle D \rangle} (F(\tau)) d\tau \|_{X^{s,b}_{\delta, \mp}} \lesssim \|F\|_{X^{s-1,b-1}_{\delta, \mp}}.
		\end{align*}
		\label{EstDuhamelCubic}
	\end{lemma}
	We provide now some additional estimates on restriction spaces $X^{s,b}_{\delta,\pm}$ that are proved using the generalized Strichartz spaces \eqref{StrichartzSpace}.
	\begin{lemma}
		Let $\frac{1}{2} < b < s < 1$. Then 
		\begin{align*}
			\| \eta(t) \int_0^t \frac{e^{\mp i (t - \tau) \langle D \rangle}}{\langle D \rangle}(F(\tau)) d\tau\|_{X^{s,b}_{\delta,\pm}} \lesssim \| F\|_{L^{\frac{2}{2-s}}_{\delta}L^{\frac{2}{2-s}}}.
		\end{align*}
		\label{lemma:MainInt}
	\end{lemma}
	\begin{proof}
		We define
		\begin{align*}
			TF =  \eta(t) \int_0^t \frac{e^{\mp i (t - \tau) \langle D \rangle}}{\langle D \rangle}(F(\tau)) d\tau.
		\end{align*}
		By Lemma \ref{EstDuhamelCubic} we deduce
		\begin{align*}
			\| TF\|_{X^{s,b}_{\delta,\pm}} \lesssim  \| F\|_{X^{s-1,b-1}_{\delta,\pm}} = (*).
		\end{align*}
		Observe that we cannot apply \eqref{EstimateStrichartzRestrictionDual} directly, given that $b-1 > -\frac{1}{2}$. Note that the inclusion $\iota_1 : L^2_{\delta}L^2 \mapsto X^{0,0}_{\delta,\pm}$ is an isometry due to Plancherel theorem. On the other hand, let $r \in (\frac{1}{2},1)$. The inclusion $\iota_2 : L^{\frac{2}{r+1}}_{\delta}L^{\frac{2}{2-r}} \mapsto X^{r-1,-\frac{1}{2}-\frac{s-b}{2(1-s)}}_{\delta,\pm}$ is continuous due to \eqref{EstimateStrichartzRestrictionDual}, since $b < s$, $s < 1$ and the exponents of the pair $(\frac{2}{r+1},\frac{2}{2-r})$ are the conjugates of the exponents of a $(1-r)$ - admissible pair. Thus, through Riesz-Thorin interpolation we obtain, for $\theta \in [0,1]$, that the inclusion
		\begin{align*}
			\iota : L^{p}_{\delta}L^{q
			} \mapsto X^{s-1,b-1}_{\delta,\pm}
		\end{align*}
		is continuous, where
		\begin{align*}
			&s-1 = \theta (r - 1), \quad b-1 = \theta\left( - \frac{1}{2} - \frac{s-b}{2(1-s)} \right), \nonumber\\
			&\frac{1}{p} = \frac{\theta(r+1)}{2} + \frac{1-\theta}{2}, \quad \frac{1}{q} = \frac{\theta (2-r)}{2} + \frac{1-\theta}{2}. 
		\end{align*}
		The second equality implies that $\theta = 2-2s$, which is a value in $(0,1)$ because $s \in (\frac{1}{2},1)$. Then $p = q = \frac{2}{2-s}$, and thus
		\begin{align*}
			(*) \lesssim  \| F\|_{L^{\frac{2}{2-s}}_{\delta}L^{\frac{2}{2-s}}}.
		\end{align*}
	\end{proof}
	\begin{lemma}
	Let $\frac12 < b < s < 1$ and $j \in \{+,-\}$. Then,
		\begin{align*}
			\|\eta(t) \int_0^t \frac{e^{\mp i (t - \tau) \langle D \rangle}}{\langle D \rangle}(u^3(\tau)) d\tau \|_{X^{s,b}_{\delta,\pm}} \lesssim \delta^{s - \frac12} \| u \|^3_{X^{s,b}_{\delta,j}}.
		\end{align*}
		 Moreover, given $\tilde{u} = (u_+,u_-) \in X^{s,b}_{\delta,+}\times X^{s,b}_{\delta,-}$, we have
		\begin{align*}
			\|\eta(t) \int_0^t \frac{e^{\mp i (t - \tau) \langle D \rangle}}{\langle D \rangle}\left( u_+ + u_- \right)^3 d\tau \|_{X^{s,b}_{\delta,\pm}} \lesssim \delta^{s - \frac12} \| \tilde{u} \|^3_{X^{s,b}_{\delta,+} \times X^{s,b}_{\delta,-}}.
		\end{align*}
		\label{lemma:DetFixedPoint}
	\end{lemma}
	\begin{proof}
		 By Lemma \ref{lemma:MainInt}, maintaining the same notation,
		\begin{equation*}
			\|  T(u^3) \|_{X^{s,b}_{\delta,\pm}} \lesssim 
			\| u^3\|_{L^{\frac{2}{2-s}}_{\delta}L^{\frac{2}{2-s}}} = \| u\|_{L^{\frac{6}{2-s}}_{\delta}L^{\frac{6}{2-s}}}^3 \leq \delta^{s - \frac12}
			\| u\|_{L^{\frac{2}{1-s}}_{\delta}L^{\frac{6}{2-s}}}^3
			\leq \delta^{s - \frac12}
			\| u\|_{X^s_{\delta}}^3 ,
		\end{equation*}
		where the penultimate inequality follows by H\"older inequality in time and the last inequality follows by the 
		fact that the pair $(\frac{2}{1-s}, \frac{6}{2-s})$ is $s$-admissible. Recalling \eqref{EstimateStrichartzRestriction}, we have proved the first inequality.
				
		Repeating the same estimates,
		\begin{align*}
			&\|\eta(t) \int_0^t \frac{e^{\mp i (t - \tau) \langle D \rangle}}{\langle D \rangle}\left( u_+ + u_- \right)^3 d\tau \|_{X^{s,b}_{\delta,\pm}} 
			\\ &
			\lesssim \delta^{s-\frac12} \|  u_+ + u_- \|_{L^{\frac{2}{1-s}}_{\delta}L^{\frac{6}{2-s}}}^3  
			\lesssim \delta^{s-\frac12} \left( \|  u_+ \|_{L^{\frac{2}{1-s}}_{\delta}L^{\frac{6}{2-s}}} +  \| u_- \|_{L^{\frac{2}{1-s}}_{\delta}L^{\frac{6}{2-s}}} \right)^3
			\\ &
			\lesssim 
			\delta^{s - \frac12} (\|  u_+ \|_{X^{s,b}_{\delta,+}} + \|  u_- \|_{X^{s,b}_{\delta,-}})^3  
			= \delta^{s - \frac12} \| \tilde{u} \|_{X^{s,b}_{\delta,+} \times X^{s,b}_{\delta,-}}^3.
		\end{align*}
	\end{proof}
	
	As an intermediate step of this Lemma \ref{lemma:DetFixedPoint}, we have proved the following corollary that will be useful later.
	\begin{corollary}
		Let $\frac12 < b < s < 1$. Then
		\begin{align*}
			\| u^3 \|_{L^{\frac{2}{2-s}}_{\delta}L^{\frac{2}{2-s}}} \lesssim \delta^{s - \frac12} \| u \|^3_{X^s_{\delta}}.
		\end{align*}
		\label{cor:EmbXsbj}
	\end{corollary}	
	The same argument leads to the following lemma. Recall that $B(0,R)$ is the ball centered
	in $0$ with radius $R$ within the space $X^{s,b}_{\delta,+} \times X^{s,b}_{\delta,-}$.
	\begin{lemma}
		Let $R>0$ and $\frac12 < b < s < 1$. Then, for any $j \in \{+,-\}$ and $u,v \in B_j(0,R)$,
		\begin{align*}
			\| \eta(t) \int_0^t \frac{e^{\mp i (t - \tau) \langle D \rangle}}{\langle D \rangle}(u^3(\tau) - v^3(\tau)) d\tau \|_{X^{s,b}_{\delta,\pm}} \lesssim R^2 \delta^{s - \frac12}\|u-v\|_{X^{s,b}_{\delta,j}},
		\end{align*}
		Moreover, given $\tilde{u}=(u_+,u_-),\tilde{v}=(v_+,v_-) \in B(0,R)$ we have
		\begin{align*}
			&\| \eta(t) \int_0^t \frac{e^{\mp i (t - \tau) \langle D \rangle}}{\langle D \rangle}( (u_+(\tau) + u_-(\tau) )^3 -  (v_+(\tau) + v_-(\tau))^3) d\tau \|_{X^{s,b}_{\delta,\pm}} \nonumber\\
			&\lesssim R^2 \delta^{s-\frac12}\|\tilde{u}-\tilde{v}\|_{X^{s,b}_{\delta,+} \times X^{s,b}_{\delta,-}}.
		\end{align*}
		\label{lemma:DetAbsorption}
	\end{lemma}
	\begin{proof}
		Applying Lemma \ref{lemma:MainInt} with its same notation for the operator $T$, we have that
		\begin{align*}
			&\| T(u^3) - T(v^3) \|_{X^{s,b}_{\delta,\pm}}  \nonumber \\
			&\lesssim \| (u - v)(u^2 + uv + v^2) \|_{L^{\frac{2}{2-s}}_{\delta}L^{\frac{2}{2-s}}} \nonumber\\
			&\lesssim  \| u-v \|_{L^{\frac{6}{2-s}}_{\delta}L^{\frac{6}{2-s}}}  (\| u \|_{L^{\frac{6}{2-s}}_{\delta}L^{\frac{6}{2-s}}}^2  + \| v \|_{L^{\frac{6}{2-s}}_{\delta}L^{\frac{6}{2-s}}}^2) = (*),
		\end{align*}
		where we used the Hölder inequality. Then we have
		\begin{align}
			&(*) \lesssim \delta^{s - \frac12}  \| u-v \|_{L^{\frac{2}{1-s}}_{\delta}L^{\frac{6}{2-s}}} (\| u \|_{L^{\frac{2}{1-s}}_{\delta}L^{\frac{6}{2-s}}}^2 + \| v \|_{L^{\frac{2}{1-s}}_{\delta}L^{\frac{6}{2-s}}}^2) \nonumber\\
			& \lesssim \delta^{s- \frac12}   \| u-v \|_{X^s_{\delta}} (\| u \|_{X^s_{\delta}}^2 + \| v \|_{X^s_{\delta}}^2) = (**).
			\label{eq:IntEst1_Absorption}
		\end{align}
		where we used H\"older inequality in time and the fact that $(\frac{2}{1-s},\frac{6}{2-s})$ is $s$-admissible. Thus the first inequality of the statement follows by \eqref{EstimateStrichartzRestriction}.		
		
		In order to prove the second inequality, we proceed as before by substituting $u$ and $v$ by $u_+ + u_-$ and $v_+ + v_-$. 
		Recalling again \eqref{EstimateStrichartzRestriction} we can bound the first factor in $(**)$ as
		\begin{align*}
			& \| (u_+ + u_-) - (v_+ + v_-) \|_{X^s_{\delta}} \leq \| u_+ - v_+ \|_{X^s_{\delta}} + \| u_- - v_- \|_{X^s_{\delta}} \nonumber\\
			&\lesssim \| u_+ - v_+ \|_{X^{s,b}_{\delta,+}} + \| u_- - v_- \|_{X^{s,b}_{\delta,-}} = \| \tilde{u} - \tilde{v} \|_{X^{s,b}_{\delta,+} \times X^{s,b}_{\delta,-}}.
		\end{align*}
		and the second factor in $(**)$ as
		\begin{align*}
			&\| u_+ + u_- \|_{X^s_{\delta}}\| v_+ + v_- \|_{X^s_{\delta}}  \leq  ( \| u_+  \|_{X^s_{\delta}} +  \| u_- \|_{X^s_{\delta}}) ( \| v_+  \|_{X^s_{\delta}} +  \| v_- \|_{X^s_{\delta}}) \nonumber\\
			&\lesssim ( \| u_+  \|_{X^{s,b}_{\delta,+}} +  \| u_- \|_{X^{s,b}_{\delta,-}}) ( \| v_+  \|_{X^{s,b}_{\delta,+}} +  \| v_- \|_{X^{s,b}_{\delta,-}}) \nonumber\\
			&= \| \tilde{u} \|_{X^{s,b}_{\delta,+} \times X^{s,b}_{\delta,-}} \| \tilde{v} \|_{X^{s,b}_{\delta,+} \times X^{s,b}_{\delta,-}},
		\end{align*}
		that concludes the proof. 
	\end{proof}
	The following corollary follows from (a small modification of) the intermediate steps of the proof above.
	\begin{corollary}
		Let $\frac12 < b < s < 1$. Then
		\begin{align*}
			\| \eta(t) \int_0^t \frac{e^{\mp i (t - \tau) \langle D \rangle}}{\langle D \rangle}
			P_{\leq N}(u^3(\tau) - v^3(\tau)) d\tau \|_{X^{s,b}_{\delta,\pm}} \lesssim \delta^{s - \frac12} M(u,v) \|u-v\|_{L^{\frac{2}{1-s}}_{\delta}L^{\frac{6}{2-s}}}
		\end{align*}
		where 
		$$
		M(u,v) : = \|u\|^2_{L^{\frac{2}{1-s}}_{\delta}L^{\frac{6}{2-s}}} + \|v\|^2_{L^{\frac{2}{1-s}}_{\delta}L^{\frac{6}{2-s}}}.
		$$
		\label{cor:LppEmb}
	\end{corollary}
	
	One can prove using generalized Strichartz estimates that \eqref{nlwcauchy} is locally well-posed for any $(u_0,u_1) \in H^s(\mathbb{T}^3) \times H^{s-1}(\mathbb{T}^3)$ if $s \geq 1/2$  (the case $s \geq 1$ is elementary), see the first Chapter in \cite{GST19}. However, for our purposes, it will be important to develop a local well-posedness theory   in restriction spaces. This is certainly a well known fact, however we will give all the details for the reader's convenience. 
	
	\begin{proposition}[Deterministic local well-posedness]
		Let  $\frac12 < b < s < 1$, $f_{\pm} \in H^s(\mathbb{T}^3)$ and $\Lambda = \max\{ \|f_+\|_{H^s(\mathbb{T}^3)} , \|f_-\|_{H^s(\mathbb{T}^3)} \}$. Then, there exists $B >0$, $C >0$ such that for all $\delta \in (0,B\Lambda^{-2/(s - \frac12)})$ the Cauchy problem
		\begin{equation}
			\left\{
			\begin{array}{l}
				(\partial_t + i\langle D \rangle)u_+ = i\langle D \rangle^{-1} \left( \frac{u_+ + u_-}{2} \right)^3, \quad u_+(x,0) = f_+(x)
				\\
				(\partial_t - i\langle D \rangle)u_- = - i\langle D \rangle^{-1} \left( \frac{u_+ + u_-}{2} \right)^3, \quad u_-(x,0) = f_-(x)
			\end{array}
			\right.
		\end{equation}
		admits a unique solution $(u_+,u_-) \in X^{s,b}_{\delta,+} \times X^{s,b}_{\delta,-}$, and it satisfies 
		\begin{align}\label{hngjfkhhjhhgfgsaskdjhfgj}
		\|(u_+,u_-)\|_{X^{s,b}_{\delta,+} \times X^{s,b}_{\delta,-}} \leq C \Lambda.
\end{align}				
		\label{prop:LWP}
	\end{proposition}
	\begin{proof}
		The proof is done by Banach fixed point. We consider $f = (f_+,f_-) \in H^s(\mathbb{T}^3) \times H^s(\mathbb{T}^3)$, and the map $\Gamma_f : X^{s,b}_{\delta,+} \times X^{s,b}_{\delta,-} \mapsto X^{s,b}_{\delta,+} \times X^{s,b}_{\delta,-}$ defined by
		\begin{align*}
			\Gamma_f(u) := 
			\begin{pmatrix} \Gamma_{f,+}(u) \\  \\ \Gamma_{f,-}(u) \\\end{pmatrix}
			:=
			\begin{pmatrix} \eta(t)e^{-it\langle D \rangle} f_+ + i \eta(t) \int_0^t \frac{e^{-i(t-\tau)\langle D \rangle}}{\langle D \rangle} \left( \frac{u_+ + u_-}{2}  \right)^3 d\tau \\  \\ \eta(t)e^{it\langle D \rangle} f_- - i \eta(t) \int_0^t \frac{e^{i(t-\tau)\langle D \rangle}}{\langle D \rangle} \left( \frac{u_+ + u_-}{2}  \right)^3 d\tau \\\end{pmatrix} 
		\end{align*}
		for any $u = (u_+,u_-) \in X^{s,b}_{\delta,+} \times X^{s,b}_{\delta,-}$. We must show that $\Gamma_f$ is a contraction. 
		
		First, we estimate
		\begin{align*}
			&\| \Gamma_{f,\pm}(u) \|_{X^{s,b}_{\delta,\pm}} \leq \| \eta(t) e^{\mp it\langle D \rangle} f_{\pm} \|_{X^{s,b}_{\delta,\pm}} + \|  \eta(t) \int_0^t \frac{e^{\mp i(t-\tau)\langle D \rangle}}{\langle D \rangle} \left( \frac{u_+ + u_-}{2}  \right)^3 d\tau \|_{X^{s,b}_{\delta,\pm}} \nonumber \\
			&\lesssim \| f_{\pm} \|_{H^s(\mathbb{T}^3)} + (*),
		\end{align*} 
		where we have used Lemma \ref{EstLinearFlowRestriction}. For the second term, we apply the second inequality of Lemma \ref{lemma:DetFixedPoint}, so that
		\begin{align*}
			(*) \lesssim \delta^{s - \frac12} \| u \|_{X^{s,b}_{\delta,+} \times X^{s,b}_{\delta,-}}^3.
		\end{align*}
		Given $u \in B(0,2 \Lambda)$ we get
		\begin{align}
			\| \Gamma_f(u) \|_{X^{s,b}_{\delta,+} \times X^{s,b}_{\delta,-}} \lesssim \Lambda + \delta^{s - \frac12}  \Lambda^3.
			\label{eq:IntEstDLWP_ball}
		\end{align}
		For the contractive property, it is enough to use the second inequality from Lemma \ref{lemma:DetAbsorption}, in a way such that, component by component and considering $u,v \in B(0,2 \Lambda)$,
		\begin{align}
			\| \Gamma_{f,\pm}(u) - \Gamma_{f,\pm}(v) \|_{X^{s,b}_{\delta,\pm}}  \lesssim \Lambda^2 \delta^{s-\frac12} \| u-v\|_{X^{s,b}_{\delta,+} \times X^{s,b}_{\delta,-}}.
			\label{eq:IntEstDLWP_cont}
		\end{align}
		Taking $\delta \in (0,B\Lambda^{-2/(s - \frac12)})$ and $B>0$ sufficiently small we have
		\begin{align*}
			\Lambda + \delta^{s- \frac12}\Lambda^3 \leq 2 \Lambda, 
			\quad \Lambda^2 \delta^{s- \frac12} \leq \frac{1}{2}.
		\end{align*}
		That concludes the proof.
	\end{proof}
	
	\section{Proof of Theorem \ref{MainThm1}}\label{Sec:Proof1}
	
	We are now ready to prove our first main result, namely Theorem \ref{MainThm1}.
	We are interested in the (local) flow associated to the equations in \eqref{nlwdec}, namely
	\begin{align}
		\begin{pmatrix} \Phi_t^+ f_+(x) \\  \\ \Phi_t^- f_-(x) \\\end{pmatrix} =  \begin{pmatrix} 
			\eta(t) e^{-it\langle D \rangle} u_{0,+} + i \eta(t) \int_0^t \frac{e^{-i(t-\tau)\langle D \rangle}}{\langle D \rangle}  \left( \frac{\Phi_{\tau}^+ f_+(x) + \Phi_{\tau}^- f_-(x)}{2} \right)^3  d\tau
			\\
			\\
			\eta(t) e^{it\langle D \rangle} u_{0,-} - i \eta(t) \int_0^t \frac{e^{i(t-\tau)\langle D \rangle}}{\langle D \rangle} \left( \frac{\Phi_{\tau}^+ f_+(x) + \Phi_{\tau}^- f_-(x)}{2} \right)^3  d\tau	\\\end{pmatrix},
		\label{ourflow_truncated}
	\end{align}
	where $(f_+,f_-) \in H^s(\mathbb{T}^3) \times H^s(\mathbb{T}^3)$.
	We have observed, in Section \ref{Sec:SOTP}, that Theorem \ref{MainThm1} can be rewritten in the following way.
	\begin{theorem}
		Let $s > 1/2$ and $f=(f_+,f_-) \in H^s(\mathbb{T}^3) \times H^s(\mathbb{T}^3)$. Then
		\begin{align*}
			\lim_{t \rightarrow 0} |\Phi^{\pm}_t f_{\pm}(x) - f_{\pm}(x)| = 0 \quad \mbox{for almost every $x \in \mathbb{T}^3$}.
		\end{align*}
		\label{thm:pwc}
	\end{theorem}
	\begin{proof}
		Without loss of generality we can restrict to consider $s \in (1/2, 1)$.
		It suffices to prove the convergence for each component of \eqref{ourflow_truncated}. 
		By Lemma~\ref{lemma:AdaptMaxEst_NL} and estimate \eqref{eq:MaxEstXsb} we only 
		need to show that
		\begin{align}
			&\lim_{N \rightarrow \infty}\|  \Phi^{\pm}_t f_{\pm}(x) - \Phi^{N,\pm}_tf_{\pm}(x) \|_{X^{s,b}_{\delta,\pm}} = 0
			\label{MaxEstimateRestriction}
		\end{align}
		for some $\delta > 0$ and $s,b >1/2$.  
		We take, in particular, $b \in (1/2, s)$, so that we can apply the preliminary results. 
		Note that
		\begin{align}\label{eq:int-detlwp-1}
			& \Phi^{\pm}_t f_{\pm}(x) - \Phi^{N,\pm}_tf_{\pm}(x) \nonumber\\ &= \eta(t) P_{>N}e^{\mp i t \langle D \rangle}f_{\pm}(x) \pm i \eta(t) P_{> N} \int_0^t \frac{e^{\mp i (t-\tau)\langle D \rangle}}{\langle D \rangle} g(x,\tau)^3 d\tau \nonumber\\
			&\pm  i \eta(t) P_{\leq N} \int_0^t \frac{e^{\mp i (t-\tau)\langle D \rangle}}{\langle D \rangle}( g(x,\tau)^3 -g_N(x,\tau)^3  ) d\tau = (i) + (ii) + (iii),
		\end{align}
		where
		\begin{align*}
			g_N(x,\tau) = \left(\frac{\Phi_{\tau}^{N,+}f_+(x) + \Phi_{\tau}^{N,-}f_-(x)}{2} \right), \quad g(x,\tau) = \left(\frac{\Phi_{\tau}^+ f_+(x) + \Phi_{\tau}^-f_-(x)}{2} \right).
		\end{align*}
		For $(i)$, thanks to Lemma \ref{EstLinearFlowRestriction} and dominated convergence theorem, we have
		\begin{align*}
			\| \eta(t) P_{>N}e^{\mp i t \langle D \rangle}f_{\pm}(x)  \|_{X^{s,b}_{\delta,\pm}} \lesssim \| P_{>N} f_{\pm} \|_{H^s} \rightarrow 0, \text{ for } N \rightarrow \infty.
		\end{align*} 
		For $(ii)$, thanks to the second inequality of Lemma \ref{lemma:DetFixedPoint}, we have that
		\begin{align*}
			&\left\| \eta(t)  P_{> N} \int_0^t \frac{e^{\mp i (t-\tau)\langle D \rangle}}{\langle D \rangle} g(x,\tau)^3 d\tau  \right\|_{X^{s,b}_{\delta,\pm}} \leq \left\| \eta(t)  \int_0^t \frac{e^{\mp i (t-\tau)\langle D \rangle}}{\langle D \rangle} g(x,\tau)^3 d\tau  \right\|_{X^{s,b}_{\delta,\pm}} \nonumber \\
			&\lesssim \delta^{s - \frac12} \| (\Phi_t^+f_+,\Phi_t^-f_-) \|_{X^{s,b}_{\delta,+} \times X^{s,b}_{\delta,-}}^3.
		\end{align*}
		Then, domination of the first integral is guaranteed since, from Proposition \ref{prop:LWP}, 
		\begin{align}
			\|(\Phi_t^+f_+,\Phi_t^-f_-) \|_ {X^{s,b}_{\delta,+} \times X^{s,b}_{\delta,-}}^3 \leq C \max\{ \|f_+\|_{H^s(\mathbb{T}^3)} , \|f_-\|_{H^s(\mathbb{T}^3)} \}^3.
			\label{eq:domination_pwc}
		\end{align}
		On the other hand, if we consider the element
		\begin{align*}
			h = \pm i \eta(t) \int_0^t \frac{e^{\mp i (t-\tau)\langle D \rangle}}{\langle D \rangle} g(x,\tau)^3 d\tau,
		\end{align*}
		we deduce that $\langle n \rangle^{s} \langle \gamma \pm \langle n \rangle\rangle^{b} \tilde{h}(n,\gamma)$ is 
		an element of $L^2_{\gamma}(\mathbb{R},\l^2_n(\mathbb{Z}^3))$. Thus, the tail
		\begin{align*}
		\sum_{|n| > N}  \int_{\mathbb{R}} \langle n \rangle^{2s} \langle \gamma \pm \langle n \rangle\rangle^{2b} |\tilde{h}(n,\gamma)|^2 d\gamma 
		\end{align*}
		converges to $0$ as $N \rightarrow \infty$. Therefore, $(ii)$ converges to $0$ in $X^{s,b}_{\delta,\pm}$ as $N \rightarrow \infty$.
		
		For $(iii)$, we apply Corollary \ref{cor:LppEmb}. Since $(\Phi_t^+f_+,\Phi_t^-f_-) \in B(0, C\Lambda)$ (recall that this is the ball in the $X^{s,b}_{\delta,+} \times X^{s,b}_{\delta,-}$ space) we have 
		\begin{align*}
			&\| \eta(t) P_{\leq N} \int_0^t \frac{e^{\mp i (t-\tau)\langle D \rangle}}{\langle D \rangle}( g(x,\tau)^3 - g_N(x,\tau)^3  ) d\tau  \|_{X^{s,b}_{\delta,\pm}} \nonumber\\
			&\lesssim \Lambda^2 \delta^{s - \frac12} \| \Phi_tf - \Phi_t^Nf \|_{X^{s,b}_{\delta,+} \times X^{s,b}_{\delta,-}}.
		\end{align*}
		Thus taking $\delta$ sufficiently small we can absorb $(iii)$ in the LHS of \eqref{eq:int-detlwp-1}. This concludes the proof.
	\end{proof}	
	
	\section{Probabilistic preliminaries}\label{Sec:ProbPre}	
	
	In this section we collect some probabilistic results that will be used in Section \ref{Sec:Proof2} in order to prove our second main Theorem \ref{MainThm2}. In particular, we adapt the construction of solutions from randomized initial data of \cite{BurqTzvet} to the 
	restriction spaces framework.
	
	We will consider initial data of the form
	\begin{align}
		f^w = \sum_{n \in \mathbb{Z}^3}g^w_n \hat{f}(n) e^{inx},
		\label{RandomIDdec}
	\end{align}
	where $f \in H^{\sigma}(\mathbb{T}^3)$, $\sigma >0$, and $(g_n^w)_{n \in \mathbb{Z}^3}$ 
	is a sequence of pairwise independent complex sub-Gaussian random variables on a probability space $(\Omega,\mathbb{P},\mathcal{G})$. 
	The sub-Gaussian variables will be of $0$-mean and unitary variance. We recall that, given a
	$0$-mean  random variable $X$, we say that $X$ is sub-Gaussian (more precisely, $b$-sub-Gaussian) if, for any $t \in \mathbb{R}$,
	\begin{align}
		\mathbb{E}e^{tX} \leq e^{b^2t^2/2}, \qquad b >0.
		\label{SG_laplace}
	\end{align}
	This characterization of sub-Gaussianity is equivalent to the existence of some $c > 0$ (related to $b$ by $c = \alpha/b^2$ for some constant $\alpha > 0$) such that, for any $\lambda > 0$,
	\begin{align}
		\mathbb{P}(|X| > \lambda) \leq 2 e^{-c\lambda^2};
		\label{SG_tail}
	\end{align}
	and to the existence of some $C>0$ such that, for any $r \geq 1$,
	\begin{align}
		\|X\|_{L^r_w(\Omega)} = (\mathbb{E}|X|^r)^{1/r} \leq C \sqrt{r} b.
		\label{SG_IneqMomentum}
	\end{align}
	In our case $b$ will be the same for all the sub-Gaussian variables considered above. 
	
	We will need the following.
	\begin{lemma}[\cite{BurqTzvet}]
		Let $d \in \mathbb{N}$, $(a_n) \in \l^2(\mathbb{Z}^d)$ and $(g^w_n)$ a sequence of complex, pairwise independent, random sub-Gaussian variables (all of them $b$-sub-Gaussian for some $b > 0$). Then, the random variable
		\begin{align}
			w \mapsto \sum_{n \in \mathbb{Z}^d} a_n g^w_n
			\label{sumSGrv}
		\end{align}
		is sub-Gaussian. In particular, there exists $C>0$ such that, for every $r \geq 1$,
		\begin{align}
			\| \sum_{n \in \mathbb{Z}^d} a_n g^w_n\|_{L^r} \leq C \sqrt{r} \|(a_n)\|_{\l^2(\mathbb{Z}^d)}.
		\end{align}
		\label{lemma:SG_KeyIneq}
	\end{lemma}
	
	Now we provide improved dispersive estimates for the linear flow evolving by randomized initial data.
	These results are standard, however we sketch the proof for the reader's convenience.
	We recall the Bernstein's inequality on $\mathbb{T}^d$:
	\begin{align}
		\| \langle D \rangle^s P_N f \|_{L^q(\mathbb{T}^d)} \lesssim N^{s + d(\frac{1}{p} - \frac{1}{q})} \|P_Nf\|_{L^p(\mathbb{T}^d)}, \quad 1 \leq p \leq q \leq \infty.
		\label{BernsteinIneq}
	\end{align}
	\begin{lemma}
		Let $\lambda, \sigma, s, t \in \mathbb{R} $, $p \in [1,\infty)$, $d \in \mathbb{N}$ and $N \in \mathbb{N}$. Then for all $\rho \in (0,d)$
		\begin{align*}
			&	\mathbb{P} [ \| P_N \langle D \rangle^s e^{\pm i t \langle D \rangle} f^w \|_{L^p_x(\mathbb{T}^{d})} > \lambda ] \lesssim e^{-c \lambda^2 N^{2(\sigma-s)}/(p \|f\|^2_{H^{\sigma}(\mathbb{T}^d)})}. 
			\\
			&	 \mathbb{P} [ \| P_N \langle D \rangle^s e^{\pm i t \langle D \rangle} f^w \|_{L^{\infty}_x(\mathbb{T}^{d})} > \lambda ]   \lesssim   e^{-c \lambda^2 N^{2(\sigma-s) - \rho} \rho /(d \|f\|^2_{H^{\sigma}(\mathbb{T}^d)})},
		\end{align*}
		for some $c > 0$. 
		In other words, given $\beta \in (0,1)$ we have that
		\begin{align*}
			&\| P_N \langle D \rangle^s e^{\pm i t \langle D \rangle} f^w \|_{L^p_x(\mathbb{T}^{d})} \lesssim (-p \log(\beta))^{1/2} N^{s-\sigma}\|f\|_{H^{\sigma}(\mathbb{T}^d)}, \nonumber\\
			&\| P_N \langle D \rangle^s e^{\pm i t \langle D \rangle} f^w \|_{L^{\infty}_x(\mathbb{T}^{d})}  \lesssim  (-d \rho^{-1} \log(\beta))^{1/2} N^{s-\sigma + \frac{\rho}{2}} \|f\|_{H^{\sigma}(\mathbb{T}^d)}
		\end{align*}
		with a probability of $1 - \beta$.
		\label{DispersiveProbEst}
	\end{lemma}
	\begin{proof}
		Note that $P_N \langle D \rangle^s e^{\pm i t \langle D \rangle} f^w$ is of the form
		\begin{align*}
			\sum_{n \in \mathbb{Z}^d} a_n g^w_n, 
		\end{align*}
		with 
		\begin{align*}
			a_n = \chi_N(n) \hat{f}(n) \langle n \rangle^s e^{inx \pm i t \langle n\rangle} ,
		\end{align*}
		where $\chi_N$ is the characteristic function of the set $\{n\in \mathbb{Z}^d : \frac{N}{2} \leq |n| \leq N\}$ in the frequency space. Let $r \geq 1$. We start considering $r \geq p$. By Minkowski inequality and Lemma \ref{lemma:SG_KeyIneq}, omitting irrelevant constants,
		\begin{align*}
			&\|  \|P_N \langle D \rangle^s e^{\pm i t \langle D \rangle} f^w\|_{L^p_x(\mathbb{T}^d)} \|_{L^r_w(\Omega)} \leq \|  \|P_N \langle D \rangle^s e^{\pm i t \langle D \rangle} f^w\|_{L^r_w(\Omega)} \|_{L^p_x(\mathbb{T}^d)} \nonumber\\
			& \lesssim C \sqrt{r} \| (a_n) \|_{\l^2 (\mathbb{Z}^d)} = C \sqrt{r} \left( \sum_{\frac{N}{2} \leq |n| \leq N} \langle n \rangle^{2s} |\hat{f}(n)|^2 \right)^{1/2} \nonumber\\
			&= C \sqrt{r} \left( \sum_{\frac{N}{2} \leq |n| \leq N} \langle n \rangle^{2(s-\sigma)} \langle n \rangle^{2\sigma} |\hat{f}(n)|^2 \right)^{1/2}  \lesssim C \sqrt{r} N^{-(\sigma-s)} \|f\|_{H^{\sigma}(\mathbb{T}^d)}.
		\end{align*}
		If $r < p$ it suffices to use Hölder's inequality in $L^r_w(\Omega)$ and follow an analogous procedure to obtain
		\begin{align*}
		&\|  \|P_N \langle D \rangle^s e^{\pm i t \langle D \rangle} f^w\|_{L^p_x(\mathbb{T}^d)} \|_{L^r_w(\Omega)} \leq &\|  \|P_N \langle D \rangle^s e^{\pm i t \langle D \rangle} f^w\|_{L^p_x(\mathbb{T}^d)} \|_{L^p_w(\Omega)} \nonumber \\
		&\lesssim C \sqrt{p} N^{-(\sigma - s)} \|f\|_{H^{\sigma}(\mathbb{T}^d)}.
		\end{align*}
		Thus, for any $r \geq 1$ we have that
		\begin{align*}
		&\|  \|P_N \langle D \rangle^s e^{\pm i t \langle D \rangle} f^w\|_{L^p_x(\mathbb{T}^d)} \|_{L^r_w(\Omega)} \lesssim C \max_{1 \leq r < p}\left\{ 1 , \sqrt{\frac{p}{r}} \right\} \sqrt{r} N^{-(\sigma - s)} \|f\|_{H^{\sigma}(\mathbb{T}^d)} \nonumber \\
		&= C \sqrt{rp} N^{-(\sigma - s)} \|f\|_{H^{\sigma}(\mathbb{T}^d)}.
		\end{align*}
		 We have just proved that
		\begin{align*}
			w \mapsto \| P_N \langle D \rangle^s e^{\pm i t \langle D \rangle} f^w\|_{L^p_x(\mathbb{T}^d)}
		\end{align*}
		is a sub-Gaussian random variable satisfying
		\begin{align*}
			\mathbb{P}[ \|P_N \langle D \rangle^s e^{\pm i t \langle D \rangle} f^w\|_{L^p_x(\mathbb{T}^d)} > \lambda] \lesssim e^{-c\lambda^2 N^{2(\sigma-s)}/(p\|f\|_{H^{\sigma}(\mathbb{T}^d)}^2)}
		\end{align*}
		for $c>0$ some constant independent of $N$. By Bernstein's inequality \eqref{BernsteinIneq}, considering any $R > 1$ and using the same argument,
		\begin{align*}
			&\|  \|P_N \langle D \rangle^s e^{\pm i t \langle D \rangle} f^w\|_{L^{\infty}_x(\mathbb{T}^d)} \|_{L^r_w(\Omega)} \lesssim N^{d/R} \|  \|P_N \langle D \rangle^s e^{\pm i t \langle D \rangle} f^w\|_{L^R_x(\mathbb{T}^d)} \|_{L^r_w(\Omega)} \nonumber\\
			& \lesssim N^{d/R} C \sqrt{r R} N^{-(\sigma - s)}\|f\|_{H^{\sigma}(\mathbb{T}^d)},
		\end{align*}
		so that 
		\begin{align*}
			\mathbb{P}[ \|P_N \langle D \rangle^s e^{\pm i t \langle D \rangle} f^w\|_{L^{\infty}_x(\mathbb{T}^d)} > \lambda] \lesssim e^{-c \lambda^2 N^{2(\sigma-s) - \frac{2d}{R}}/(R \|f\|_{H^{\sigma}(\mathbb{T}^d)}^2)}
		\end{align*}
		where $R$ is as large as desired. For the last part, regarding the $p < \infty$ case, it suffices to take 
		\begin{align*}
			\beta = e^{-c \lambda^2 N^{2(\sigma-s)}/(p\|f\|_{H^{\sigma}(\mathbb{T}^d)}^2)} \Longleftrightarrow \lambda = (-c^{-1}p \log(\beta))^{1/2} N^{s-\sigma}\|f\|_{H^{\sigma}(\mathbb{T}^d)},
		\end{align*}
		and regarding $p = \infty$,
		\begin{align*}
			\beta = e^{-c \lambda^2 N^{2(\sigma-s) - \frac{2d}{R}}/(R\|f\|_{H^{\sigma}(\mathbb{T}^d)}^2)} \Longleftrightarrow \lambda = (-c^{-1} R \log(\beta))^{1/2} N^{s-\sigma + \frac{d}{R}}\|f\|_{H^{\sigma}(\mathbb{T}^d)}.
		\end{align*}
		This concludes the proof once we define $\rho = 2d / R$.
	\end{proof}
	In an almost identical way, we provide probabilistic Strichartz estimates.
	\begin{lemma}
		Let $\lambda, \sigma, s, t \in \mathbb{R} $, $p \in [1,\infty)$, $d \in \mathbb{N}$ and $N \in \mathbb{N}$. Then for all $\rho \in (0,d)$
		\begin{align*}
			&\mathbb{P} [ \| P_N \langle D \rangle^s e^{\pm i t \langle D \rangle} f^w \|_{L^p_{t,x}([0,1] \times \mathbb{T}^{d})} > \lambda ] \lesssim e^{-c \lambda^2 N^{2(\sigma-s)}/(p \|f\|^2_{H^{\sigma}(\mathbb{T}^d)})}, \nonumber\\
			& \mathbb{P} [ \| P_N \langle D \rangle^s e^{\pm i t \langle D \rangle} f^w \|_{L^{\infty}_{t,x}([0,1] \times \mathbb{T}^{d})} > \lambda ] \lesssim e^{-c \lambda^2 N^{2(\sigma-s) - \rho} \rho /(d \|f\|^2_{H^{\sigma}(\mathbb{T}^d)})},
		\end{align*}
		for some $c > 0$. 
		In other words, given $\beta \in (0,1)$ we have that
		\begin{align*}
			&\| P_N \langle D \rangle^s e^{\pm i t \langle D \rangle} f^w \|_{L^p_{t,x}([0,1] \times \mathbb{T}^{d})} \lesssim (-p \log(\beta))^{1/2} N^{s-\sigma}\|f\|_{H^{\sigma}(\mathbb{T}^d)}, \nonumber\\
			&\| P_N \langle D \rangle^s e^{\pm i t \langle D \rangle} f^w \|_{L^{\infty}_{t,x}([0,1] \times \mathbb{T}^{d})}  \lesssim  (-d \rho^{-1}\log(\beta))^{1/2} N^{s-\sigma + \frac{\rho}{2}}\|f\|_{H^{\sigma}(\mathbb{T}^d)}
		\end{align*}
		with a probability of $1 - \beta$.
		\label{StrichartzProbEst}
	\end{lemma}
	
	A first consequence of Lemmata \ref{StrichartzProbEst} and \ref{DispersiveProbEst} is that, as long as $s < \sigma$, then $e^{\pm i t \langle D \rangle}f^w \in W^{s,p}(\mathbb{T}^d)$ $w$-almost surely, for all $p \in [1, \infty)$. Indeed, considering $\beta \in (0,1)$ and $N$ to be dyadic scales, then for any $p \in [1,\infty)$ 
	\begin{align}
		&\| e^{\pm i t \langle D \rangle} f^w \|_{W^{s,p}(\mathbb{T}^d)} \leq \sum_N \|P_N \langle D \rangle^s e^{\pm i t \langle D \rangle} f^w\|_{L^p_x(\mathbb{T}^d)} \nonumber\\
		&\lesssim \sqrt{p} \|f\|_{H^{\sigma}(\mathbb{T}^d)} \sum_{k \in \mathbb{N}} (-\log(\beta_k))^{1/2} 2^{-k(\sigma-s)}
		\label{eq:fwGralProbBound}
	\end{align}
	with probability $1 - \beta_k$ for each of the terms. We need to select $(\beta_k)_{k \in \mathbb{N}}$ such that
	\begin{align}\label{eq:keysum-stickingprobs}
	\sum_{k \in \mathbb{N}} \beta_k \leq \beta \text{ and } \sum_{k \in \mathbb{N}} (-\log(\beta_k))^{1/2} 2^{-k(\sigma-s)} \lesssim (-\log(\beta))^{1/2}. 
	\end{align}
	It suffices to take $\beta_k = 2^{-k}\beta$.	Therefore, it can be proved that, for any $t \in \mathbb{R}$ and $w$-almost surely,
	\begin{align*}
		e^{\pm i t \langle D \rangle} f^w \in \bigcap_{s < \sigma} W^{s,p}(\mathbb{T}^d).
	\end{align*}
	Similarly, using the $L^\infty$ estimate from Lemma \ref{DispersiveProbEst} we have, for any $t \in \mathbb{R}$ and $w$-almost surely,
	\begin{align*}
		e^{\pm i t \langle D \rangle} f^w \in \bigcap_{s < \sigma} C^s(\mathbb{T}^d).
	\end{align*}
	We also give some additional bounds that will be useful later. Let $N \in \mathbb{N}$ and consider the same values of $\beta_k$, $s < \sigma$ and \eqref{eq:keysum-stickingprobs}. If $N > \frac{1}{\beta}$, then
	\begin{align*}
	\| P_{> N} e^{it \langle D \rangle} f^w\|_{W^{s,p}(\mathbb{T}^d)} \lesssim \sqrt{p} \|f\|_{H^{\sigma}(\mathbb{T}^d)} \sum_{k\in \mathbb{N}, k > \log_2(N)} \left(k\log\left(2\right)\right)^{1/2} 2^{-k(\sigma-s)}.
	\end{align*}
	with probability $\geq 1-\beta$. Let $\varepsilon > 0$. The last sum is convergent, so there exists some $N_0 \in \mathbb{N}$ such that, if $N > N_0$, then it will be less than $\frac{\varepsilon C}{\sqrt{p} \|f\|_{H^{\sigma}(\mathbb{T}^d)}}$ for $C > 0$ some suitable constant. All in all, if $N > \max\{N_0,\frac{1}{\beta}\}$ we have that
	\begin{align*}
	\| P_{> N} e^{it \langle D \rangle} f^w\|_{W^{s,p}(\mathbb{T}^d)} < \varepsilon
	\end{align*}
	with probability $\geq 1 - \beta$. Applying Bernstein's inequality, we equally obtain that there exists some $N_1 \in \mathbb{N}$ such that, if $N > \max\{N_1,\frac{1}{\beta}\}$, we have
	\begin{align}\label{eq:probcontrol-largefreq-inf}
	\| P_{> N} e^{it \langle D \rangle} f^w\|_{W^{s,\infty}(\mathbb{T}^d)} < \varepsilon
	\end{align}
	with probability $\geq 1-\beta$.
	
	The next result provides $w$-a.s. uniform convergence for the linear flow.
	\begin{lemma}[Uniform convergence a.s. for the linear Klein-Gordon equation in $\mathbb{T}^d$] 
		Let $d \in \mathbb{N}$, $\sigma > 0$ and $f^w$ the randomized initial datum from \eqref{RandomIDdec}, given $f \in H^{\sigma}(\mathbb{T}^d)$. Then, 
		$w$-almost surely
		\begin{align*}
			\sup_{x \in \mathbb{T}^d} |e^{\pm i t \langle D \rangle}f^w(x) - f^w(x)| \rightarrow 0, \text{ as } t \rightarrow 0.
		\end{align*}
		\label{RandomUniformPCL}
	\end{lemma}
	\begin{proof}
	Let $\varepsilon > 0$, and define the event
		\begin{align*}
			A_{\varepsilon} = \{w \in \Omega : \limsup_{t \rightarrow 0} \| e^{\pm i t \langle D \rangle}f^w - f^w\|_{L^{\infty}_x(\mathbb{T}^d)} > \varepsilon \}.
		\end{align*}
		Note that if we prove that $\mathbb{P}(A_{\varepsilon}) = 0$ for any $\varepsilon > 0$, then we would be done. Given $N \in \mathbb{N}$, decompose
		\begin{align*}
			|e^{\pm i t \langle D \rangle}f^w - f^w| \leq |P_{>N}f^w| + |P_{>N}e^{\pm i t \langle D \rangle}f^w| + |P_{\leq N} (e^{\pm i t \langle D \rangle}f^w - f^w)|. 
		\end{align*}
		Let $\beta \in (0,1)$. Regarding the first two terms, from \eqref{eq:probcontrol-largefreq-inf} there exists some $N_{\varepsilon} \in \mathbb{N}$ such that, if $N > \max\{N_{\varepsilon},\frac{2}{\beta}\}$, then
		\begin{align*}
		\| P_{>N}e^{\pm i t \langle D \rangle}f^w\|_{L^{\infty}_x(\mathbb{T}^d)} < \frac{\varepsilon}{3}
		\end{align*}
		for any $t \in \mathbb{R}$ with probability $ \geq 1 - \frac{\beta}{2}$.
		\\		
		Regarding the last term, 
		\begin{align}
			&\| P_{\leq N} (e^{\pm i t \langle D \rangle}f^w - f^w) \|_{L^{\infty}_x(\mathbb{T}^d)} = \left\| \sum_{|n| \leq N} (e^{\pm i t \langle n \rangle} - 1) e^{inx} \widehat{f^w}(n) \frac{\langle n \rangle^{s^*}}{\langle n \rangle^{s^*}} \right\|_ {L^{\infty}_x(\mathbb{T}^d)} \nonumber \\
			&\leq \sup_{|n| \leq N} |e^{\pm i t \langle n \rangle} - 1| \left( \sum_{n \in \mathbb{Z}^d} \langle n \rangle^{-2s^*} \right)^{1/2} \left( \sum_{|n| \leq N} \langle n \rangle^{2s^*} |\widehat{f^w}(n)|^2 \right)^{1/2} \nonumber\\
			&\lesssim_{s^*} |t|N^{s^*+1} \|P_{\leq N} f^w\|_{L^2_x(\mathbb{T}^d)},
			\label{eq:Int1_UniformLinearConv}
		\end{align}
		where we applied Cauchy-Schwarz in the first inequality, the mean value theorem in the second one, and we took $s^* > d/2$ in order to have convergence of the series of general term $\langle n \rangle^{-2s^*}$. From \eqref{eq:fwGralProbBound} we know that 
$\|P_{\leq N} f^w\|_{L^2_x(\mathbb{T}^d)}$ is $w$-a.s. finite independently of $N$. Then
		\begin{align*}
		\limsup_{t \rightarrow 0} |t|N^{s^*+1} \|P_{\leq N} f^w\|_{L^2_x(\mathbb{T}^d)} = 0
		\end{align*}
		for almost every $w$. All in all, 
		\begin{align*}
		&\mathbb{P}(A_{\varepsilon}) \leq \mathbb{P}(w \in \Omega : \limsup_{t \rightarrow 0}\| P_{>N}e^{\pm i t \langle D \rangle}f^w\|_{L^{\infty}_x(\mathbb{T}^d)} > \frac{\varepsilon}{3}) \nonumber \\
		& + \mathbb{P}(w \in \Omega :\| P_{>N}f^w\|_{L^{\infty}_x(\mathbb{T}^d)} > \frac{\varepsilon}{3}) \leq \beta
		\end{align*}
		for any $\beta \in (0,1)$. Since $\beta$ can be taken as close to $0$ as wanted, then $\mathbb{P}(A_{\varepsilon}) = 0.$
	\end{proof}
	
	
	We now construct solutions with initial data randomized as in \eqref{RandomIDdec}, starting by initial 
	data in $H^{\sigma}(\mathbb{T}^3) \times H^{\sigma}(\mathbb{T}^3)$, with $\sigma > 0$.  Thus we adapt the result from \cite{BurqTzvet}
	to the restriction space framework.
	
	\begin{theorem}\label{Prop:BT}
		Let	$\frac12 < b < s < 1$, $\sigma \in (0, s)$ and $f = (f_+,f_-) \in H^{\sigma}(\mathbb{T}^3) \times H^{\sigma}(\mathbb{T}^3)$. We denote with $f^w = (f^w_+,f^w_-)$ the corresponding randomization defined as in \eqref{RandomIDdec}, and $\Lambda = \max \{\| f_+ \|_{H^{\sigma}(\mathbb{T}^3)},\| f_- \|_{H^{\sigma}(\mathbb{T}^3)}\}$. Consider the Cauchy problem
		\begin{equation}
			\left\{
			\begin{array}{l}
				(\partial_t + i\langle D \rangle)u^w_+ = i\langle D \rangle^{-1} \left( \frac{u^w_+ + u^w_-}{2} \right)^3, \quad u^w_+(x,0) = f^w_+(x)
				\\
				(\partial_t - i\langle D \rangle)u^w_- = - i\langle D \rangle^{-1} \left( \frac{u^w_+ + u^w_-}{2} \right)^3, \quad u^w_-(x,0) = f^w_-(x).
			\end{array}
			\right.
			\label{eq:CauchyProb}
		\end{equation}
		Let $\delta \in (0,1)$ be sufficiently small. Then, with a probability $\geq 1-e^{- \Lambda^{-2} \delta^{\frac13(\frac12 - s)}}$, the Cauchy problem \eqref{eq:CauchyProb} admits a unique solution $(u^w_+,u^w_-)$ in the unit ball of the Banach space $Y^{s,b,\sigma}_{\delta,+} \times Y^{s,b,\sigma}_{\delta,-}$, where
		\begin{align*}
			&Y^{s,b,\sigma}_{\delta,\pm} = \{ e^{\mp i t \langle D \rangle}f^w + h: h \in X^{s,b}_{\delta,\pm}  \}, \nonumber\\
			&\| e^{\mp i t \langle D \rangle}f^w + h \|_{Y^{s,b,\sigma}_{\delta,\pm}} =  \| h \|_{X^{s,b}_{\delta,\pm} }.
		\end{align*} 
		Moreover, \eqref{eq:CauchyProb} admits $w$-almost surely a $\delta_w \in (0,1)$ such that a unique solution $(u^w_+,u^w_-)$ exists in the unit ball of $Y^{s,b,\sigma}_{\delta_w,+} \times Y^{s,b,\sigma}_{\delta_w,-}$.	\label{prop:ProbLWP}
	\end{theorem}
	
		\begin{remark}
	We are interested in Theorem \ref{prop:ProbLWP} for $\sigma$ close to zero, in which case the statement implies (taking $s$ close to $1$) a $1 - \varepsilon$ smoothing for the Duhamel contribution. One can modify the argument to show that the same amount of smoothing can be obtained for larger values of $\sigma$, however this generalisation is not necessary for our purposes.  
	\end{remark}

	\begin{proof}
		Following \cite{BurqTzvet}, letting 
		$$
		u^{w}_{j} =  z^{w}_{j} + v_{j}, \qquad j \in  \{ +, - \},
		$$
		we rewrite \eqref{eq:CauchyProb} as
		\begin{align*}
			&(\partial_t \pm i \langle D\rangle)z^w_{\pm} = 0, \nonumber\\
			&(z^w_+(x,0),z^w_-(x,0)) = (f^w_+(x),f^w_-(x)), 
		\end{align*}
		and
		\begin{align*}
			&(\partial_t \pm i\langle D\rangle)v_{\pm} = \pm i \langle D \rangle \left( \frac{(z^w_++v_+) + (z^w_- + v_-)}{2}  \right)^3, \nonumber\\
			&(v_+(x,0),v_-(x,0)) = (0,0).
		\end{align*}
		Denoting $v= (v_{+}, v_{-} )$, 
		the proof reduces to proving that the map $\Gamma_{f^w}$ defined by
		\begin{align*} 
			&\Gamma_{f^w}(v) :=
			\begin{pmatrix} \Gamma_{f^w,+}(v) \\  \\ \Gamma_{f^w,-}(v) \\\end{pmatrix} :=
			\begin{pmatrix} i \eta(t) \int_0^t \frac{e^{-i(t-\tau)\langle D \rangle}}{\langle D \rangle} 
				\left( \frac{(z^w_+ + v_+) + (z^w_-+v_-)}{2}  \right)^3 d\tau \\  \\  - i \eta(t) \int_0^t \frac{e^{i(t-\tau)\langle D \rangle}}{\langle D \rangle} \left( \frac{(z^w_+ + v_+) + (z^w_- + v _-)}{2}  \right)^3 d\tau \\\end{pmatrix} 
		\end{align*}
		is a contraction on the ball $B(0,1)$ of $X^{s,b}_{\delta,+} \times X^{s,b}_{\delta,-}$ (with $\delta \in (0,1)$ small enough) except 
		for initial data in an exceptional set of measure $\leq e^{-\Lambda^{-2} \delta^{\frac13 (\frac12 - s)}}$. 
		This would provide the almost sure existence of a unique solution of the form
		$$(u^w_+,u^w_-) = (z^w_+ + v_+,z^w_-+v_-).$$ 
Indeed, given $\delta \in (0,1)$ sufficiently small denote by $A_{\delta}$ a subset of $\Omega$ for which
\begin{itemize}
\item given $w \in A_ {\delta}$ and the randomized initial data $f^w$, there exits a unique solution $(u^w_+,u^w_-)$ in the unit ball of $Y^{s,b,\sigma}_{\delta,+} \times Y^{s,b,\sigma}_{\delta,-}$ with $(u^w_+(x,0),u^w_-(x,0)) = (f^w_+(x),f^w_-(x))$,
\item $\mathbb{P}(A_\delta^c) \leq  e^{-\Lambda^{-2} \delta^{\frac13 (\frac12 - s)}}$.
\end{itemize}
Let
\begin{align*}
A = \bigcup_{n \in \mathbb{N}, n \geq n_0} A_{1/n}
\end{align*}
for $n_0$ large enough so that for any $w \in A$ there exists a unique solution in $Y^{s,b,\sigma}_{\delta_w,+} \times Y^{s,b,\sigma}_{\delta_w,-}$, for some $\delta_w > 0$. Moreover, for any $k \in \mathbb{N}$, $k \geq n_0$,
\begin{align*}
\mathbb{P}(A^c) = \mathbb{P}\left(\bigcap_{n \in \mathbb{N}, n\geq n_0} A_{1/n}^c \right) \leq \mathbb{P}(A_{1/k}^c) \leq  e^{-\Lambda^{-2} k^{\frac13 (s - \frac12)}},
\end{align*}
so taking $k \rightarrow \infty$ and recalling that $s > \frac{1}{2}$, we have that $\mathbb{P}(A) = 1$.\\

		Regarding the quantitative part of the statement, applying Lemma \ref{lemma:MainInt} and Hölder's inequality we get
		\begin{align*}
			&\| \Gamma_{f^w,\pm}(v) \|_{X^{s,b}_{\delta,\pm}} \lesssim \| (z^w_++v_+ + z^w_-+v_-)^3 \|_{L^{\frac{2}{2-s}}_{\delta}L^{\frac{2}{2-s}}} \nonumber\\
			& \lesssim \delta^{s-\frac{1}{2}} \| z^w_+ + z^w_- \|_{L^{\frac{2}{1-s}}_{\delta}L^{\frac{6}{2-s}}}^3 + \| (v_+ + v_-)^3 \|_{L^{\frac{2}{2-s}}_{\delta}L^{\frac{2}{2-s}}} = \delta^{s - \frac{1}{2}}  (i) + (ii).
		\end{align*}
		
		For $(i)$, denoting $l_s = \max\{\frac{2}{1-s},\frac{6}{2-s}\}$ and using \eqref{eq:fwGralProbBound} with $\beta$ small enough we get
		\begin{align}
			&(i) \lesssim \| z^w_+ + z^w_- \|_{L^{l_s}_{\delta} L^{l_s}}^3 \lesssim (\| f_+ \|_{H^{\sigma}(\mathbb{T}^3)}^3 + \| f_- \|_{H^{\sigma}(\mathbb{T}^3)}^3) (-\log(\beta))^{3/2} \nonumber\\
			&\leq  \Lambda^3 (-\log(\beta))^{3/2}  =:  \lambda^3
			\label{eq:IntEstPLWP_prob}
		\end{align}
		with probability $1-\beta$, for $\beta \in (0,1)$ (we are interested in small values of $\beta$). 
		
		For $(ii)$, we apply the Corollary \ref{cor:EmbXsbj} and \eqref{EstimateStrichartzRestriction}
		\begin{align*}
			&(ii) \lesssim  \delta^{s-\frac12} \| v_+ + v_- \|_{X^s_\delta}^3 \lesssim \delta^{s-\frac12} (\| v_+\|_{X^{s,b}_{\delta,+}} + \| v_-\|_{X^{s,b}_{\delta,-}})^3 
			= \delta^{s-\frac12} \| v \|_{X^{s,b}_{\delta,+} \times X^{s,b}_{\delta,-}}^3.
		\end{align*}
		This implies that
		\begin{align}
			\| \Gamma_{f^w,\pm}(v) \|_{X^{s,b}_{\delta,\pm}} \leq C \delta^{s - \frac12} ( \lambda^3 + 1 )
			\label{eq:IntEstPLWP_ball}
		\end{align}
		as long as $v \in B(0,1)$. 
		
		For the contractive property, we consider $v_1 = (v_{1,+},v_{1,-})$ and $v_2 = (v_{2,+},v_{2,-})$ in  $B(0,1)$ and do
		\begin{align*}
			&\| \Gamma_{f^w,\pm}(v_1) - \Gamma_{f^w,\pm}(v_2) \|_{X^{s,b}_{\delta,\pm}} = \left\| \eta(t) \int_0^t \frac{e^{\mp i(t-\tau)\langle D \rangle}}{\langle D \rangle} \left( \left( \frac{g^w_1}{2} \right)^3 - \left( \frac{g^w_2}{2} \right)^3  \right) d\tau \right\|_{X^{s,b}_{\delta,\pm}} \nonumber\\
			&= (iii),
		\end{align*}
		where
		\begin{align*}
			g^w_i  = z^w_+ + v_{i,+} + z^w_- + v_{i,-}, \quad i \in \{1,2\}.
		\end{align*}
		Recalling that $\frac12 < b < s < 1$, we can apply Lemma \ref{lemma:DetAbsorption}:
		\begin{align}
			&(iii) \leq C \delta^{s-\frac12} \| v_1 - v_2 \|_{X^{s,b}_{\delta,+} \times X^{s,b}_{\delta,-}}.
			\label{eq:IntEstPLWP_cont}
		\end{align}
		Thus the map is a contraction on $B(0,1)$ if
		\begin{enumerate}
			\item[(1)] Recall \eqref{eq:IntEstPLWP_ball} :
			\begin{align*}
				C \delta^{s-\frac12} ( \lambda^3 + 1 ) \leq 1.
			\end{align*}
			\item[(2)] Recall \eqref{eq:IntEstPLWP_cont} :
			\begin{align*}
				C \delta^{s - \frac12}  \leq \frac{1}{2}.
			\end{align*}
		\end{enumerate}
		Recalling $\Lambda^3 (-\log(\beta))^{3/2}  =:  \lambda^3$, we see that if we choose $\lambda = \delta^{\frac16(\frac12 - s)}$, then the conditions are satisfied for all sufficiently small $\delta >0$ as long as the initial data are outside an exceptional set of measure 
		$\leq e^{-\frac{\lambda^2}{\Lambda^2}} = e^{- \Lambda^{-2} \delta^{\frac13 (\frac12 - s)}}$, that completes the proof.
	\end{proof}
	\section{Proof of Theorem \ref{MainThm2}}\label{Sec:Proof2}
	We are now ready to prove our second main result, namely Theorem \ref{MainThm2}. This will require a combination 
	of the nonlinear smoothing effect of the Duhamel integral from Section \ref{Sec:DetPre} and of the probabilistic Strichartz estimates from Section \ref{Sec:ProbPre}. The maximal estimates framework developed in Section \ref{sec:TMATNPC} is crucial in order to minimize the amount of nonlinear smoothing that is necessary in order to prove the theorem with minimal regularity. In the same spirit of the deterministic case, we rewrite Theorem \ref{MainThm2} as follows
	\begin{theorem}
		Let  $\sigma > 0$ and $f = (f_+,f_-) \in H^{\sigma}(\mathbb{T}^3) \times H^{\sigma}(\mathbb{T}^3)$. Then, using the same notation than Theorem \ref{prop:ProbLWP},
		\begin{align*}
			\lim_{t \to 0}
			|\Phi^{\pm}_t f^w_{\pm}(t,x) - f^w_{\pm}(x)| = 0, \qquad \mbox{for almost every $x \in \mathbb{T}^3$},
		\end{align*}
		$w$-almost surely.
		\label{thm:Ppwc}
	\end{theorem}
	Thus we will focus on the proof of this last result.
	\begin{proof}
		From Theorem \ref{prop:ProbLWP}, we know that $w$-almost surely we can select a $\delta_w \in (0,1)$ such that the Cauchy problem \eqref{eq:CauchyProb} admits a unique solution $\Phi_t f^w$ in $Y^{s,b,\sigma}_{\delta_w,+} \times Y^{s,b,\sigma}_{\delta_w,-}$. 
		Moreover, from now on we fix $\frac12 < b < s < 1$. Strictly speaking, this is not enough to prove the statement in full generality. However the low regularity case ($s$ close to $1/2$) is the hardest and more interesting, thus in order to avoid making the presentation too technical we will only focus on it. 
		
		Invoking Lemma \ref{lemma:AdaptMaxEst_NL} we have reduced the problem to prove the validity of 
		\begin{align*}
			\lim_{N \rightarrow \infty} \| \sup_{0 < t < \delta_w} | \Phi^{\pm}_t f^w_{\pm} - \Phi^{N,\pm}_t f^w_{\pm} | \hspace{1mm} \|_{L^2(\mathbb{T}^3)}  = 0,
		\end{align*}
		$w$-almost surely.
		We decompose the difference as
		\begin{align}
			&| \Phi^{\pm}_t f^w_{\pm}- \Phi^{N,\pm}_t f^w_{\pm} | \nonumber\\
			&\leq |P_{>N} e^{\mp i t \langle D \rangle} f^w_{\pm}| + | (\Phi^{\pm}_t f^w_{\pm} - e^{\mp i t \langle D \rangle} f^w_{\pm}) - (\Phi^{N,\pm}_t f^w_{\pm} - P_{\leq N} e^{\mp i t \langle D \rangle} f^w_{\pm}) |.
			\label{eq:1decomp_Probpwc}
		\end{align}
		Then
		\begin{align*}
		&\| \sup_{0 < t < \delta_w} | \Phi^{\pm}_t f^w_{\pm} - \Phi^{N,\pm}_t f^w_{\pm} | \hspace{1mm} \|_{L^2(\mathbb{T}^3)} \leq \| P_{>N} e^{\mp i t \langle D \rangle} f^w_{\pm} \|_{L^{\infty}([0,\delta_w] \times \mathbb{T}^3)} \nonumber \\
		&+ \left\| \sup_{0 < t < \delta_w} \left| (\Phi^{\pm}_t f^w_{\pm} - e^{\mp i t \langle D \rangle} f^w_{\pm}) - (\Phi^{N,\pm}_t f^w_{\pm} - P_{\leq N} e^{\mp i t \langle D \rangle} f^w_{\pm}) \right| \right\|_{L^2(\mathbb{T}^3)} \nonumber \\
		&= (i) + (ii).
		\end{align*} 
		Let $\lambda > 0$. From Lemma \ref{DispersiveProbEst} one can prove that there exists some constant $c>0$ and $\rho \in (0,2\sigma)$ such that
		\begin{align*}
		\mathbb{P}(  \| P_{> N} e^{\mp i t \langle D\rangle} f^w_{\pm} \|_{L^{\infty}(\mathbb{T}^3)} > \lambda ) \lesssim e^{-c \lambda^2 N^{2\sigma - \rho} \rho/(d \|f_{\pm}\|^2_{H^{\sigma}(\mathbb{T}^3)})}.
		\end{align*}
		Thus, thanks to Borel-Cantelli Lemma we have that, $w$-almost surely, there exists some $N_{\lambda} \in \mathbb{N}$ such that, for any $N > N_{\lambda}$,
		\begin{align*}
		\| P_{> N} e^{\mp i t \langle D\rangle} f^w_{\pm} \|_{L^{\infty}(\mathbb{T}^3)} \leq \lambda.
		\end{align*}
		In other words, $w$-almost surely
		\begin{equation}\label{mfkdslkdmnfsdkldfm}
			\lim_{N \to \infty} \| P_{> N} e^{\mp i t \langle D\rangle} f^w_{\pm} \|_{L^{\infty}(\mathbb{T}^3)} =0.
		\end{equation}
		Thus we have reduced the problem to prove the $w$-almost sure convergence to $0$ related to $(ii)$, that is
		\begin{align*}
			\lim_{N \rightarrow \infty} \| \sup_{0 < t < \delta_w} \left| v_{\pm} - v_{N,\pm} | \hspace{1mm} \right\|_{L^2(\mathbb{T}^3)}  = 0
		\end{align*}
		where
		\begin{align*}
			v_{N,\pm} = \Phi^{N,\pm}_t f^w_{\pm} - P_{\leq N} e^{\mp i t \langle D \rangle} f^w_{\pm}, \quad v_{\pm} = v_{\infty,\pm}.
		\end{align*}
		Recall that we are considering $s > \frac{1}{2}$ and $b > \frac{1}{2}$, thus the embedding \eqref{eq:MaxEstXsb} reduces further the problem to show
		\begin{align}\label{eq:12epsilon-smoothing}
			\lim_{N \rightarrow \infty} \|  v_{\pm} - v_{N,\pm} \|_{X^{s,b}_{\delta_w,\pm}}  = 0
		\end{align}
		$w$-almost surely. 
		
		To prove so, we decompose:
		\begin{align}
			&v_{\pm} - v_{N,\pm} \nonumber \\
			&= \pm \frac{i \eta(t)}{8} \int_0^t \frac{e^{\mp i (t-\tau) \langle D\rangle}}{\langle D\rangle}((\Phi^{+}_\tau f^w_{+} + \Phi^{-}_\tau f^w_{-})^3 - P_{\leq N}(\Phi^{N,+}_\tau f^w_{+} + \Phi^{N,-}_\tau f^w_{-})^3) d \tau
			\label{eq:diff_vvn}
		\end{align}
		and we further decompose
		\begin{align*}
			&(\Phi^{+}_\tau f^w_{+} + \Phi^{-}_\tau f^w_{-})^3 - P_{\leq N}(\Phi^{N,+}_\tau f^w_{+} + \Phi^{N,-}_\tau f^w_{-})^3 = I_{N} (\tau)+ J_{N} (\tau),
		\end{align*}
		where 
		\begin{align}
			&I_{N} (\tau) := P_{\leq N} \left( (P_{\leq N}e^{- i \tau \langle D\rangle} f^w_{+} + v_{+} + P_{\leq N}e^{i \tau \langle D\rangle} f^w_{-} + v_{-})^3 \right) \nonumber \\
			&- P_{\leq N} \left(( P_{\leq N}e^{-i \tau \langle D\rangle} f^w_{+} + v_{N,+} + P_{\leq N}e^{i \tau \langle D\rangle} f^w_{-} + v_{N,-} )^3  \right)
			\label{eq:mainterm_ProbPwc}
		\end{align}
		and
		\begin{align}
			&J_{N} (\tau) := P_{\leq N} \left( ( e^{- i \tau \langle D\rangle} f^w_{+} + v_{+} + e^{i \tau \langle D\rangle} f^w_{-} + v_{-})^3 \right) \nonumber \\
			& - P_{\leq N}  \left( ( P_{\leq N}e^{- i \tau \langle D\rangle} f^w_{+} + v_{+} + P_{\leq N}e^{ i \tau \langle D\rangle} f^w_{-} + v_{-} )^3   \right) \nonumber \\
			&+ P_{> N} (\Phi^{+}_{\tau} f^w_{+} + \Phi^{-}_{\tau} f^w_{-})^3.
			\label{eq:remainders_ProbPwc}
		\end{align}
		Regarding the contribution of $ P_{> N} (\Phi^{+}_{\tau} f^w_{+} + \Phi^{-}_{\tau} f^w_{-})^3$ to  \eqref{eq:diff_vvn}, we note that 
		\begin{align*}
			\lim_{N \rightarrow \infty}\left\|\eta(t) \int_0^t \frac{e^{\mp i (t-\tau) \langle D\rangle}}{\langle D\rangle}(P_{> N}(\Phi^{+}_{\tau} f^w_{+} + \Phi^{-}_{\tau} f^w_{-})^3) d \tau \right\|_{X^{s,b}_{\delta_w,\pm}} = 0
		\end{align*}
		by dominated convergence and Lemma \ref{lemma:DetFixedPoint} since 
		\begin{align*}
			&\left\|\eta(t) \int_0^t \frac{e^{\mp i (t-\tau) \langle D\rangle}}{\langle D\rangle}( (\Phi^{+}_{\tau} f^w_{+} + \Phi^{-}_{\tau} f^w_{-})^3) d \tau \right\|_{X^{s,b}_{\delta_w,\pm}} \nonumber \\
			&\lesssim \delta_w^{s- \frac12} \| (\Phi^{+}_t f^w_{+},\Phi^{-}_t f^w_{-})\|_{X^{s,b}_{\delta_w,+} \times X^{s,b}_{\delta_w,-}}^3; 
		\end{align*}
		and the right hand side is $w$-almost surely finite (see Theorem~\ref{Prop:BT}).
		
		For the other contribution of $J_{N}$ to  \eqref{eq:diff_vvn}, we apply Corollary \ref{cor:LppEmb}:
		\begin{align}
			&\left\|\eta(t) \int_0^t \frac{e^{\mp i (t-\tau) \langle D\rangle}}{\langle D\rangle}\Big(  J_N(\tau) - P_{> N}(\Phi^{+}_{\tau} f^w_{+} + \Phi^{-}_{\tau} f^w_{-})^3 \Big) d \tau \right\|_{X^{s,b}_{\delta_w,\pm}} \nonumber\\
			&\lesssim  R_{\omega}^2 \delta_w^{s- \frac12} \| P_{> N} (e^{- i t \langle D \rangle} f^w_{+} + e^{ i t \langle D \rangle} f^w_{-}) \|_{L^{\frac{2}{1-s}}_{\delta_w}L^{\frac{6}{2-s}}}, \label{mfkdasldkjnfgkdslgmk}
		\end{align}
		where 
		$$
		R_{w}^2 := \| e^{- i t \langle D\rangle} f^w_{+} + e^{ i t \langle D\rangle} f^w_{-} \|_{L^{\frac{2}{1-s}}_{\delta_w}L^{\frac{6}{2-s}}}^2  + \| v_{+} + v _{-} \|_{L^{\frac{2}{1-s}}_{\delta_w}L^{\frac{6}{2-s}}}^2.
		$$
		In the same way that we proved \eqref{mfkdslkdmnfsdkldfm} 
		we can prove that $w$-almost surely
		$$
		\lim_{N \to \infty}	\| P_{> N} e^{\mp i t \langle D \rangle} f^w_{\pm} \|_{L^{\frac{2}{1-s}}_{\delta_w}L^{\frac{6}{2-s}}} =0.
		$$
		At the same time, invoking Lemma \ref{StrichartzProbEst} again we have that 
		$$\| e^{\pm i t \langle D\rangle} f^w_{\pm} \|_{L^{\frac{2}{1-s}}_{\delta_w}L^{\frac{6}{2-s}}}$$ is $w$-almost surely finite. 
		Moreover, by the embedding \eqref{EstimateStrichartzRestriction}  we have
		$$ 
		\| v_{\pm} \|_{L^{\frac{2}{1-s}}_{\delta_w}L^{\frac{6}{2-s}}} 
		\lesssim   \| v_{\pm} \|_{X^{s,b}_{\delta_{w}, \pm}}  $$
		and the right hand side is $w$-almost surely finite by Theorem \ref{Prop:BT}.  
		Thus $R_{w}$ is $w$-almost surely finite and \eqref{mfkdasldkjnfgkdslgmk} goes to zero as $N \to \infty$, $w$-almost surely.
		
		In conclusion, we have proved that $w$-almost surely
		$$
		\lim_{N \to \infty} \| \eta(t) \int_0^t \frac{e^{\mp i (t-\tau) \langle D\rangle}}{\langle D\rangle} J_N (\tau) d\tau \|_{X^{s,b}_{\delta_w,\pm}} = 0.
		$$
		Thus, taking $X^{s,b}_{\delta_w,\pm}$ norm of \eqref{eq:diff_vvn} we arrive $w$-almost surely to 
		\begin{align}
			\|		v_{\pm} - v_{N,\pm} \|_{X^{s,b}_{\delta_w,\pm}} \lesssim 
			\| \eta(t) \int_0^t \frac{e^{\mp i (t-\tau) \langle D\rangle}}{\langle D\rangle}I_{N}(\tau) d \tau \|_{X^{s,b}_{\delta_w,\pm}}
			+ o_N(1).
		\end{align}
		
		We will show that the first term on the right hand side can be absorbed into the 
		left hand side, so the proof is concluded. 
		In order to do so we apply Corollary \ref{cor:LppEmb} again 
		\begin{align}\label{eq:Pball-lebesgueest}					
			&\left\|\eta(t) \int_0^t \frac{e^{\mp i (t-\tau) \langle D\rangle}}{\langle D\rangle}  I_{N} ( \tau) d \tau \right\|_{X^{s,b}_{\delta_w,\pm}} \nonumber\\
			&= \left\|\eta(t) \int_0^t \frac{e^{\mp i (t-\tau) \langle D\rangle}}{\langle D\rangle} P_{\leq N} \left( (P_{\leq N}e^{- i \tau \langle D\rangle} f^w_{+} + v_{+} + P_{\leq N}e^{i \tau \langle D\rangle} f^w_{-} + v_{-})^3 \right. \right. \nonumber \\
			&\left. \left. -  ( P_{\leq N}e^{-i \tau \langle D\rangle} f^w_{+} + v_{N,+} + P_{\leq N}e^{i \tau \langle D\rangle} f^w_{-} + v_{N,-} )^3  \right) d\tau \right\|_{X^{s,b}_{\delta_w,\pm}} \nonumber\\ 
			&\leq C  (R_{w,1}^2 + R_{w,2}^2)  \delta_w^{{s - \frac12}}  \|  v_{+} + v_{-} - v_{N,+} - v_{N,-} \|_{L^{\frac{2}{1-s}}_{\delta_w}L^{\frac{6}{2-s}}}  \nonumber\\
		\end{align}
		for some constant $C > 0$, where 
		\begin{align*}
		&R_{w,1}^2 :=  \| e^{- i t \langle D\rangle} f^w_{+} + e^{ i t \langle D\rangle} f^w_{-} \|_{L^{\frac{2}{1-s}}_{\delta_w}L^{\frac{6}{2-s}}}^2 , \nonumber \\ 
		 &R_{w,2}^2 :=
		\|  v_{N,+} + v_{N,-} \|_{L^{\frac{2}{1-s}}_{\delta_w}L^{\frac{6}{2-s}}}^2  + \| v_{+} + v_{-} \|_{L^{\frac{2}{1-s}}_{\delta_w}L^{\frac{6}{2-s}}}^2.
		\end{align*}
		
		Now, proceeding as in the proof of Theorem \ref{Prop:BT} we see that
		$w$-almost surely we have  $C R_{w,2}^2  \delta_w^{s-\frac12} \leq 1/4$ for all $\delta_w$ small enough. Thus, recalling also \eqref{eq:fwGralProbBound}, 
		it is easy to see that $w$-almost surely we can produce a (small) $\delta_w$ such that  
		$$
		C (R_{w,1}^2 + R_{w,2}^2)  \delta_w^{{s - \frac12}} \leq 1/2, 
		$$
		that concludes the proof.
	\end{proof}
	
	\section*{Acknowledgments}
	The authors are supported by the Basque Government through the program BERC 2022-2025 (BCAM), by the project PID2021-123034NB-I00
	funded by MCIN/ AEI /10.13039/501100011033 and by the Severo Ochoa accreditation CEX2021-001142-S (BCAM). RL is also supported by the Ramon y Cajal fellowship RYC2021-031981-I. PM is also supported by the predoctoral program of the Education Department of the Basque Government.


\begin{thebibliography}{10}
		
		
		
		\bibitem{MR1194782}
		J.~Bourgain.
		\newblock A remark on {S}chr\"{o}dinger operators.
		\newblock {\em Israel J. Math.}, 77(1-2):1--16, 1992.
		
		\bibitem{MR1209299}
		J.~Bourgain.
		\newblock Fourier transform restriction phenomena for certain lattice subsets
		and applications to nonlinear evolution equations. {I}. {S}chr\"{o}dinger
		equations.
		\newblock {\em Geom. Funct. Anal.}, 3(2):107--156, 1993.
		
		\bibitem{Bourgain2013}
		J.~Bourgain.
		\newblock On the {S}chr\"{o}dinger maximal function in higher dimension.
		\newblock {\em Tr. Mat. Inst. Steklova}, 280(Ortogonal'nye Ryady, Teoriya
		Priblizheni\u{\i} i Smezhnye Voprosy):53--66, 2013.
		
		\bibitem{Bourgain2016}
		J.~Bourgain.
		\newblock A note on the {S}chr\"{o}dinger maximal function.
		\newblock {\em J. Anal. Math.}, 130:393--396, 2016.
		
		
				\bibitem{BDNY24} B. Bringmann, Y. Deng, A. Nahmod and H. Yue. 
		\newblock Invariant Gibbs measures for the three dimensional cubic nonlinear wave equation. 
		\newblock{\em Invent. Math.}, 236, no. 3, 1133--1411, no 24.



		\bibitem{BurqTzvet}
		N.~Burq and N.~Tzvetkov.
		\newblock Random data {C}auchy theory for supercritical wave equations. {I}.
		{L}ocal theory.
		\newblock {\em Invent. Math.}, 173(3):449--475, 2008.
		
		
		\bibitem{Carleson}
		L.~Carleson.
		\newblock Some analytic problems related to statistical mechanics.
		\newblock In {\em Euclidean harmonic analysis ({P}roc. {S}em., {U}niv.
			{M}aryland, {C}ollege {P}ark, {M}d., 1979)}, volume 779 of {\em Lecture Notes
			in Math.}, pages 5--45. Springer, Berlin, 1980.
		
		
		
		\bibitem{CSL}
		E. Compaan, R. Luc\`a, G. Staffilani. 
		\newblock Pointwise Convergence of the Schr\"odinger Flow. 
		\newblock {\em Int. Math.
			Res. Not.}, 1, 596--647, 2021.
		
		
		
		\bibitem{C832}
		M.~G. Cowling.
		\newblock Harmonic Analysis on Semigroups.
		\newblock In volume 117(2) of {\em Ann. of Math. (2)}, pages 267--283, 1983.
		
		
		\bibitem{C831}
		M.~G. Cowling.
		\newblock Pointwise behavior of solutions to {S}chr\"{o}dinger equations.
		\newblock In {\em Harmonic analysis ({C}ortona, 1982)}, volume 992 of {\em
			Lecture Notes in Math.}, pages 83--90. Springer, Berlin, 1983.
		
		
		\bibitem{DahlbergKenig}
		B.~E.~J. Dahlberg and C.~E. Kenig.
		\newblock A note on the almost everywhere behavior of solutions to the
		{S}chr\"{o}dinger equation.
		\newblock In {\em Harmonic analysis ({M}inneapolis, {M}inn., 1981)}, volume 908
		of {\em Lecture Notes in Math.}, pages 205--209. Springer, Berlin-New York,
		1982.
		
		\bibitem{1608.07640}
		C.~Demeter and S.~Guo.
		\newblock Schr\"odinger maximal function estimates via the pseudoconformal
		transformation, 2016. Preprint on arXiv:1608.07640.
		
		
		\bibitem{Du2019}
		X. Du and R. Zhang.
		\newblock Sharp $l^2$ estimates of the Schr\"{o}dinger maximal function in
		higher dimensions.
		\newblock {\em Ann. of Math.}, 189(3):837, 2019.
		
		\bibitem{DGL}
		X.~Du, L.~Guth, and X.~Li.
		\newblock A sharp {S}chr\"{o}dinger maximal estimate in {$\Bbb R^2$}.
		\newblock {\em Ann. of Math.}, 186(2):607--640, 2017.
		
		\bibitem{MR3842310}
		X.~Du, L.~Guth, X.~Li, and R.~Zhang.
		\newblock Pointwise convergence of {S}chr\"{o}dinger solutions and multilinear
		refined {S}trichartz estimates.
		\newblock {\em Forum Math. Sigma}, 6:e14, 18, 2018.
		
		\bibitem{EL21}
		D.~Eceizabarrena and R.~Luc\`a. 
		\newblock Convergence over fractals for the periodic Schr\"odinger equation. 
		\newblock {\em Anal.
			PDE}, 15(7):1775--1805, 2022.
		
		
		\bibitem{ETBook}
		M.~B. Erdoğan and N.~Tzirakis.
		\newblock {\em  Dispersive Partial Differential Equations: Wellposedness and
			Applications.}.
		\newblock London Mathematical Society Student Texts. Cambridge University
		Press, 2016.
		
		\bibitem{G90}
		M.~G. Grillakis.
		\newblock Regularity and Asymptotic Behavior of the Wave Equation with a Critical Nonlinearity.
		\newblock {\em Ann. of Math.}, 132(3):485--509,
		1990.
		
		\bibitem{GST19}
		M.~Gubinelli, P.~E.~Souganidis, and N.~Tzvetkov.
		\newblock In {\em Singular random dynamics ({C}etraro, 2016)} .
		\newblock {\em Lecture Notes in Math.}, CIME Foundation Subseries, pages 221--313. Springer, 2019.
		
		
		
		\bibitem{MR2264734}
		S.~Lee.
		\newblock On pointwise convergence of the solutions to {S}chr\"{o}dinger
		equations in {$\Bbb R^2$}.
		\newblock {\em Int. Math. Res. Not.}, pages Art. ID 32597, 21, 2006.
		
		\bibitem{LP21}
		R. Luc\`a and F. Ponce-Vanegas. 
		\newblock Convergence over fractals for the Schr\"odinger equation. 
		To appear on {\em Indiana Univ. Math. J.}. Preprint on arXiv: 2101.02495.
		
		\bibitem{Luc2018}
		R.~Luc{\`{a}} and K.~M. Rogers.
		\newblock Average decay of the fourier transform of measures with applications.
		\newblock {\em J. Eur. Math. Soc (JEMS)}, 21(2):465--506,
		Oct. 2018.
		
		\bibitem{MR3613507}
		R.~Luc\`a and K.~M. Rogers.
		\newblock Coherence on fractals versus pointwise convergence for the
		{S}chr\"{o}dinger equation.
		\newblock {\em Comm. Math. Phys.}, 351(1):341--359, 2017.
		
		\bibitem{MR3903115}
		R.~Luc\`a and K.~M. Rogers.
		\newblock A note on pointwise convergence for the {S}chr\"{o}dinger equation.
		\newblock {\em Math. Proc. Cambridge Philos. Soc.}, 166(2):209--218, 2019.
		
		\bibitem{LuhrMend}
		J.~L\"{u}hrmann and D.~Mendelson.
		\newblock Random data {C}auchy theory for nonlinear wave equations of
		power-type on {$\Bbb {R}^3$}.
		\newblock {\em Comm. Partial Differential Equations}, 39(12):2262--2283, 2014.
		
		\bibitem{MR1671214}
		A.~Moyua, A.~Vargas, and L.~Vega.
		\newblock Restriction theorems and maximal operators related to oscillatory
		integrals in {$\mathbb{R}^3$}.
		\newblock {\em Duke Math. J.}, 96(3):547--574, 1999.
		
		\bibitem{MR2409184}
		A.~Moyua and L.~Vega.
		\newblock Bounds for the maximal function associated to periodic solutions of
		one-dimensional dispersive equations.
		\newblock {\em Bull. Lond. Math. Soc.}, 40(1):117--128, 2008.
		
		\bibitem{OP16}
		T.~Oh and O.~Pocovnicu.
		\newblock Probabilistic global well-posedness of the energy-critical defocusing quintic nonlinear wave equation on $\mathbb{R}^3$.
		\newblock {\em J. Math. Pures Appl.}, 105:342--366, 2016.
		
		\bibitem{OP17}
		T.~Oh and O.~Pocovnicu.
		\newblock A remark on almost sure global well-posedness of the energy-critical defocusing nonlinear wave equations in the periodic setting.
		\newblock {\em Tohoku Math. J. (2)}, 69(3):455--481, 2017.
		
		\bibitem{MR904948}
		P.~Sj\"{o}lin.
		\newblock Regularity of solutions to the {S}chr\"{o}dinger equation.
		\newblock {\em Duke Math. J.}, 55(3):699--715, 1987.
		
		
		\bibitem{SX15}
		C.~Sun and B.~Xia.
		\newblock  Probabilistic well-posedness for supercritical wave equation on $\mathbb{T}^3$, 2015. Preprint on arXiv:1508.00228.

		
		\bibitem{MR2033842}
		T.~Tao.
		\newblock A sharp bilinear restrictions estimate for paraboloids.
		\newblock {\em Geom. Funct. Anal.}, 13(6):1359--1384, 2003.
		
		\bibitem{tao2006nonlinear}
		T.~Tao.
		\newblock {\em Nonlinear Dispersive Equations: Local and Global Analysis}.
		\newblock Conference Board of the Mathematical Sciences. Regional conference
		series in mathematics. American Mathematical Society, 2006.
		
		\bibitem{MR1748920}
		T.~Tao and A.~Vargas.
		\newblock A bilinear approach to cone multipliers. {I}. {R}estriction
		estimates.
		\newblock {\em Geom. Funct. Anal.}, 10(1):185--215, 2000.
		
		\bibitem{MR1748921}
		T.~Tao and A.~Vargas.
		\newblock A bilinear approach to cone multipliers. {II}. {A}pplications.
		\newblock {\em Geom. Funct. Anal.}, 10(1):216--258, 2000.
		
		\bibitem{TTNonlin}
		L. Thomann and N. Tzvetkov. 
		\newblock Gibbs measure for the periodic derivative nonlinear Schr\"odinger equation.
		\newblock {\em Nonlinearity} 23(11), 2771--2791, 2010.
		
		\bibitem{Vega}
		L.~Vega.
		\newblock Schr\"{o}dinger equations: pointwise convergence to the initial data.
		\newblock {\em Proc. Amer. Math. Soc.}, 102(4):874--878, 1988.
		
		\bibitem{Walther}\label{Walther}
		B. G. Walther, Some $L^p(L^{\infty})$ and $L^2(L^{2})$-estimates for oscillatory Fourier transforms,
		in Analysis of divergence (Orono, MA, 1997), 213–231, {\em Appl. Numer. Harmon. Anal.},
		Birkh\"auser, Boston, MA.
		
		\bibitem{WangZhang}
		X.~Wang and C.~Zhang.
		\newblock Pointwise convergence of solutions to the Schrödinger equation on
		manifolds.
		\newblock {\em Canad. J. Math.}, 71(4):983–--995, 2019.
		
		
	\end{thebibliography}
\end{document}